\documentclass{amsart}

\usepackage{pgf,tikz,pgfplots}
\pgfplotsset{compat=1.15}
\usepackage{mathrsfs}
\usetikzlibrary{arrows}

\usepackage[T1]{fontenc}
\usepackage{microtype}
\usepackage{lmodern}
\usepackage[colorlinks=true,urlcolor=blue, citecolor=red,linkcolor=blue,linktocpage,pdfpagelabels, bookmarksnumbered,bookmarksopen]{hyperref}
\usepackage[hyperpageref]{backref}
\usepackage{amsthm} %for citing inside theorem header
\usepackage{latexsym,amsmath,amssymb}

\usepackage{accents}
\usepackage{esint}

\usepackage{soul}
\usepackage{mathtools} %For \chi with a much lower index (\mychi)
\usepackage{xparse} %For a new definiton of \chi with a slightly lower index
\usepackage[capitalize]{cleveref}

\usepackage[shortlabels]{enumitem}

\usepackage[textsize=small]{todonotes}
\title[Uniqueness for the Half-wave equation]{On uniqueness for half-wave maps in dimension \texorpdfstring{$\lowercase{d} \geq 3$}{d >= 3}}

\renewcommand{\S}{{\mathbb S}}

\newtheorem{theorem}{Theorem}
\newtheorem{lemma}[theorem]{Lemma}
\newtheorem{corollary}[theorem]{Corollary}

\theoremstyle{definition}

\theoremstyle{remark}

\newcommand{\R}{\mathbb{R}}

\newcommand{\Z}{\mathbb{Z}}

\newcommand{\brac}[1]{\left (#1 \right )}
\newcommand{\norm}[1]{\left \|#1 \right \|}
\newcommand{\scpr}[2]{\left \langle #1 , #2\right \rangle}
\newcommand{\abs}[1]{\left\lvert #1 \right \rvert}

\renewcommand{\vec}[1]{{\bf #1}}

\newcommand{\barint}{
\rule[.036in]{.12in}{.009in}\kern-.16in \displaystyle\int }

\newcommand{\barcal}{\text{$ \rule[.036in]{.11in}{.007in}\kern-.128in\int $}}

\def\mvint_#1{\mathchoice
          {\mathop{\vrule width 6pt height 3 pt depth -2.5pt
                  \kern -8pt \intop}\nolimits_{\kern -3pt #1}}%
          {\mathop{\vrule width 5pt height 3 pt depth -2.6pt
                  \kern -6pt \intop}\nolimits_{#1}}%
          {\mathop{\vrule width 5pt height 3 pt depth -2.6pt
                  \kern -6pt \intop}\nolimits_{#1}}%
          {\mathop{\vrule width 5pt height 3 pt depth -2.6pt
                  \kern -6pt \intop}\nolimits_{#1}}}

\numberwithin{theorem}{section} \numberwithin{equation}{section}

\newcommand{\lap}{\Delta }
\newcommand{\aleq}{\lesssim}
\newcommand{\ageq}{\succsim}
\newcommand{\aeq}{\approx}

\newcommand{\Rz}{\mathcal{R}}
\newcommand{\laps}[1]{(-\Delta)^{\frac{#1}{2}}}
\newcommand{\Ds}[1]{|\nabla|^{#1}}
\newcommand{\Dso}{|\nabla|}

\newcommand{\lapms}[1]{\mathscr{I}_{#1}}

\usepackage{scalerel}[2014/03/10]
\usepackage[usestackEOL]{stackengine}
\def\avint{\,\ThisStyle{\ensurestackMath{%
			\stackinset{c}{.2\LMpt}{c}{.5\LMpt}{\SavedStyle-}{\SavedStyle\phantom{\int}}}%
		\setbox0=\hbox{$\SavedStyle\int\,$}\kern-\wd0}\int}
\renewcommand{\div}{\operatorname{div}}

\let\latexchi\chi
\makeatletter
\renewcommand\chi{\@ifnextchar_\sub@chi\latexchi}
\newcommand{\sub@chi}[2]{% #1 is _, #2 is the subscript
  \@ifnextchar^{\subsup@chi{#2}}{\latexchi^{}_{#2}}%
}
\newcommand{\subsup@chi}[3]{% #1 is the subscript, #2 is ^, #3 is the superscript
  \latexchi_{#1}^{#3}%
}
\makeatother

\author{Eugene Eyeson}
\email[Eugene Eyeson]{eue3@pitt.edu}

\author{Silvino Reyes Farina}
\email[Silvino Reyes Farina]{sir25@pitt.edu}

\author{Armin Schikorra}
\email[Armin Schikorra]{armin@pitt.edu}
\address{Department of Mathematics,
University of Pittsburgh,
301 Thackeray Hall,
Pittsburgh, PA 15260, USA}

\begin{document}
\begin{abstract}
Extending an argument by Shatah and Struwe we obtain uniqueness for solutions of the half-wave map equation in dimension $d \geq 3$ in the natural energy class.
\end{abstract}

.
\maketitle 
\tableofcontents

\section{Introduction and main result}
Half-wave maps appear in the physics literature as the continuum limit of Calagero-Moser spin systems, see \cite{LenzmannSok20} and references within. 
They are solutions $\vec{u}: \R^d \times [0,T] \to \S^2 \subset \R^3$ to the half-wave maps equation which is given by
\begin{equation}\label{eq:half-wavemapseq}
 \partial_t \vec{u}  = \vec{u} \wedge \laps{1} \vec{u} \quad \text{in $\R^d \times (0,T)$}.
\end{equation}
Here and henceforth $\wedge$ denotes the cross product in $\R^3$ and $\laps{1} = \Dso{}$ is the half-Laplacian.
Recently, several authors, e.g. \cite{KS18,LS2018,GL18,Berntson_2020,KK21,SWZ21,Liu2021}, began to study mathematical properties of \eqref{eq:half-wavemapseq}. 

In \cite{KS18} the structural relation of the half-wave map equation to the wave map equation
\[
 \partial_{tt} \vec{u} - \lap \vec{u} = \brac{- \partial_{t}\vec{u} \cdot \partial_{t} \vec{u}+ \nabla \vec{u} \cdot \nabla \vec{u}} \vec{u}
\]
was discovered and exploited. Namely, by a direct computation, see \cite[p.663]{KS18}, a solution of \eqref{eq:half-wavemapseq} solves 
\begin{equation}\label{eq:solks}
\begin{split}
 \partial_{tt} \vec{u} - \lap \vec{u} =& \vec{u} |\nabla \vec{u}|^2  - \vec{u} |\Dso{}\vec{u}|^2\\
 &+ \Dso{} \vec{u}\, \brac{\scpr{\vec{u}}{\Dso{} \vec{u}}}\\
 &+ \vec{u} \wedge \Dso{} \brac{ \vec{u} \wedge\Dso{} \vec{u}}- \vec{u} \wedge \brac{\vec{u} \wedge (-\lap) \vec{u}}.
 \end{split}
\end{equation}
% \[
%  \begin{split}
%  \partial_{tt} \vec{u} - \lap \vec{u} =& - \vec{u} \Dso{}\scpr{\vec{u}}{\Dso{} \vec{u}} + \vec{u} (\nabla \vec{u} \cdot \nabla \vec{u})\\
%  &+ \Dso{} \vec{u}\, \brac{\scpr{\vec{u}}{\Dso{} \vec{u}}}\\
%  &+ \vec{u} \wedge \Dso{} (\vec{u} \wedge\Dso{} \vec{u}) - \vec{u} \wedge (\vec{u} \wedge (-\lap) \vec{u}).
%  \end{split}
% \]
% 
The authors of \cite{KS18} then raised the question if one can use this route to extend methods developed for wave maps, e.g. those in the celebrated articles \cite{T98,Tao01I,Tao01II}, to half-wave maps. Following this principle, in \cite{KS18,KK21,Liu2021} different well-posedness results for large dimensions were discovered. Observe that the energy-critical dimension for the halfwave map equation is $d=1$, as opposed to the energy-critical dimension of the wave map equation, which is $d=2$.

In this work we also follow this spirit of treating solutions to the halfwave map equation as solutions to a wave-map-type equation, but we focus on techniques developed for wave maps by Shatah and Struwe \cite{SS02}. Our main result is the following uniqueness property of half-wave maps.

\begin{theorem}[Uniqueness]\label{th:main}
Let $d \geq 3$ and $\alpha \in (1,d+\frac{1}{2})$. If $\vec{u},\vec{v}: \R^d \times [0,T]  \to \S^2$ are smooth solutions to the half-wave map equation with the same initial data $\vec{u}(\cdot,0) = \vec{v}(\cdot,0) \in Q+ C_c^\infty(\R^d,\R^3)$ for some $Q \in \S^2$, and if
\begin{equation}\label{eq:energyassumption}
\|\Ds{\alpha}\vec{u}\|_{L^2_t L^{(\frac{2d}{2\alpha-1},2)}_x(\R^d\times (0,T))} +  \|\Ds{\alpha}\vec{v}\|_{L^2_t L^{(\frac{2d}{2\alpha-1},2)}_x(\R^d\times (0,T))} < \infty
\end{equation}
then $\vec{u} \equiv \vec{v}$.
\end{theorem}
Here $L^{(p,q)}$ denotes the Lorentz space. The a priori assumptions \eqref{eq:energyassumption} are the natural energy assumptions for initial data $\vec{u_0},\vec{v_0} \in \dot{H}^{\frac{d}{2}}(\R^d)$, which was one of the crucial observations in \cite{SS02} where Shatah and Struwe observed this for $\alpha = 1$. A careful inspection of their argument actually gives the assumption \eqref{eq:energyassumption} for small $\alpha > 1$, see \Cref{s:strichartz}. 

As in the case of Shatah-Struwe, our arguments rely mostly on geometric properties combined with fractional Leibniz rules and related commutator estimates. However, while for the wave map equation the proof of uniqueness fits on one page, our argument does not -- since it relies on several further structural observations of the ``tangential part'' of the right-hand side of \eqref{eq:solks}, which we hope are of independent interest.

\subsection*{Outline}
In \Cref{s:prelim} we introduce operators and estimates needed in the proof of \Cref{th:main}. We believe that most, if not all, of these estimates are known at least to some experts -- and they can be proven by standard techniques. In \Cref{s:decay} we discuss the main part of the proof, the decay estimates in time, \Cref{th:detest}. While we are substantially inspired by the argument by Shatah-Struwe, our estimates are more elaborate, even though they mostly rely on the fractional Leibniz rule. The decay estimates of \Cref{th:detest} combined the standard Gr\"onwall type inequality imply \Cref{th:main}, see \Cref{s:proofmain}. In \Cref{s:strichartz} we discuss the suitability of the assumptions \eqref{eq:energyassumption} for $\alpha > 1$, $\alpha \aeq 1$.

We believe that our arguments can also be used to discuss existence for small data in the above energy class as in Shatah-Struwe, which will be the subject of a future investigations.

\subsection*{Acknowledgements}
Funding by NSF Career DMS-2044898 and Simons foundation grant no 579261 is gratefully acknowledged.

\section{Preliminaries: Leibniz Rule, Sobolev embedding and Gagliardo-Nirenberg}\label{s:prelim}
Throughout the paper we use the standard $\aleq$, $\ageq$, $\aeq$ notation: we write $A \aleq B$ if there is a multiplicative constant $C > 0$,  which may change from line to line, such that $A \leq C B$. We write $A \aeq B$ if $A \aleq B$ and $B \aleq A$.

We denote vectors in bold-face, such as $\vec{v} \in \R^3$.

The fractional Laplacian is as a multiplier operator via the Fourier transform $\mathcal{F}$ for a constant $c > 0$,
\[
 \Ds{s} f(x) \equiv \laps{s} f(x) := \mathcal{F}^{-1} (c|\cdot|^s \mathcal{F} f(\cdot)) (x).
\]
We also remark the useful potential representation for some (different) constant $c \in \R$ and $s \in (0,1)$
\[
 \Ds{s} f(x) \equiv \laps{s} f(x) = c\int_{\R^d} \frac{f(x)-f(y)}{|x-y|^{n+s}}\, dy,
\]
and, for $s \in (0,2)$,
\[
 \Ds{s} f(x) \equiv \laps{s} f(x) = -\frac{c}{2}\int_{\R^d} \frac{f(x+h)+f(x-h)-2f(x)}{|h|^{d+s}}\, dh.
\]
As for negative powers, $\lapms{s} \equiv (-\lap)^{-\frac{s}{2}}$ denotes the Riesz potential,
\[
\lapms{s} f(x) \equiv (-\lap)^{-\frac{s}{2}} f(x) := \mathcal{F}^{-1} (c|\cdot|^{-s} \mathcal{F} f(\cdot)) (x).
\]
It has the potential representation for $s \in (0,d)$,
\[
 \lapms{s} f(x) \equiv (-\lap)^{-\frac{s}{2}} f(x) = c \int_{\R^d} |x-z|^{s-n} f(z)\, dz.
\]

Some of our arguments will depend on Lorentz space estimate, $L^{p,q}(\R^d)$. We only recall the main properties and refer the reader to \cite[Section 1.4]{GrafakosClassical}: For $p \in (1,\infty)$ we have $L^{p,p}(\R^d) = L^p(\R^d)$, $L^{p,q_1}(\R^d) \subset L^{p,q_2}(\R^d)$ whenever $q_1 \leq q_2$, $q_1,q_2 \in [1,\infty]$. $L^{p,\infty}(\R^d)$ is often referred to as the weak $L^p$-space.

\subsection{Embedding theorems}
A casual observation we will use throughout this paper is the following comparability
\begin{lemma}\label{la:nablalaps}
For any $p \in (1,\infty)$,
\[
 \|\nabla f\|_{L^p(\R^n)} \aeq \|\Ds{1} f\|_{L^{p}(\R^n)}.
\]
More generally in the realm of Lorentz spaces, for any $q \in [1,\infty]$
\[
 \|\nabla f\|_{L^{p,q}(\R^n)} \aeq \|\Ds{1} f\|_{L^{p,q}(\R^n)}.
\]
\end{lemma}
\begin{proof}
This follows since the Riesz transforms $\Rz_i := \partial_i \lapms{1}$ are bounded operators on $L^p(\R^n) \to L^p(\R^n)$ and $L^{p,q}(\R^n) \to L^{p,q}(\R^n)$ for any $p \in (1,\infty)$ and $q \in [1,\infty]$, combined with the following facts that can easily be checked using the Fourier transform,
\[
 \partial_i = c_1 \Rz_i \Ds{1}, \quad \text{and} \quad \Ds{1} = c_2 \sum_{i=1}^n \Rz^i \partial_i.
\]
\end{proof}

\begin{lemma}[Sobolev inequality]\label{la:sobolev}
Let $\alpha \in (0,d)$ and $p \in (1,\frac{d}{\alpha})$ then for any $f \in C_c^\infty(\R^d)$,
\[
 \|f\|_{L^{\frac{dp}{d-\alpha p}}(\R^d)} \aleq \|\Ds{\alpha} f\|_{L^p(\R^d)}.
\]
Equivalently, in terms of the Riesz potential $\lapms{\alpha} \equiv \Ds{-\alpha}$ we have 
\[
 \|\lapms{\alpha} f\|_{L^{\frac{dp}{d-\alpha p}}(\R^d)} \aleq \|f\|_{L^p(\R^d)}.
\]
In terms of Lorentz spaces we have for any $q \in [1,\infty]$,
\[
 \|f\|_{L^{\frac{dp}{d-\alpha p},q}(\R^d)} \aleq \|\Ds{\alpha} f\|_{L^{p,q}(\R^d)}.
\]
and
\[
 \|\lapms{\alpha} f\|_{L^{\frac{dp}{d-\alpha p,q}}(\R^d)} \aleq \|f\|_{L^{p,q}(\R^d)}.
\]
An important limit version is 
\begin{equation}\label{eq:limitsobLinfty}
 \|\lapms{\alpha} f\|_{L^{\infty}(\R^d)} \aleq \|f\|_{L^{\frac{d}{\alpha},1}(\R^d)}.
\end{equation}
\end{lemma}
All these estimates are consequence of Young's convolution inequality in Lorentz spaces, \cite[Lemma 4.8]{Hunt66} or \cite[Theorem 1.4.25.]{GrafakosClassical} combined with interpolation, using that $\lapms{\alpha} f = c |\cdot|^{\alpha-d} \ast f$ and that $|\cdot|^{\alpha-d} \in L^{\frac{d}{d-\alpha},\infty}(\R^d)$.

\begin{lemma}[Gagliardo-Nirenberg inequality]\label{la:GagliardoNirenberg}
For $\alpha \in (0,1)$, $p \in (1,\infty)$, we have
\[
 \|\Ds{\beta} f\|_{L^{\frac{p}{\beta}}(\R^d)} \aleq \|f\|_{L^\infty}^{1-\beta}\, \|\Dso f\|_{L^{p}(\R^d)}^\beta.
\]
\end{lemma}
\begin{proof}
By \cite[Lemma 1, p.329]{Runst86} we have
\[
 \|\Ds{\beta} f\|_{L^{\frac{p}{\beta}}(\R^d)} \aeq \|f\|_{\dot{F}^{\beta}_{\frac{p}{\beta},2}} 
 \aleq \|f\|_{L^\infty(\R^d)}^{1-\beta} \|f\|_{W^{1,p}(\R^d)}^\beta.
\]
Here $\dot{F}$ denotes the homogeneous Triebel-Lizorkin space, cf. \cite{RS96}. 
Applying this result to $f(\lambda \cdot)$ and taking $\lambda \to \infty$ we conclude that 
\[
 \|\Ds{\beta} f\|_{L^{\frac{p}{\beta}}(\R^d)} \aleq \|f\|_{L^\infty(\R^d)}^{1-\beta} \|\nabla f\|_{L^p(\R^d)}^\beta \aeq \|f\|_{L^\infty(\R^d)}^{1-\beta} \|\Dso f\|_{L^p(\R^d)}^\beta 
\]
\end{proof}
For convenience we record a special case of the previous inequality.
\begin{corollary}[Gagliardo-Nirenberg-Sobolev inequality]\label{co:gagnirsob}
Assume 
\begin{enumerate}
 \item $\beta \in (0,\frac{1}{2}]$ and $p \in [\frac{2d}{\beta},\infty)$, or
 \item $\beta \in (\frac{1}{2},1]$ and $p \in [\frac{2d}{\beta},\frac{2d}{2\beta-1}]$.
 \end{enumerate}
 Then for $\theta = 2\brac{\beta - \frac{d}{p}} \in [\beta,1]$ we have
\[ 
 \|\Ds{\beta} f\|_{L^{p}(\R^d)} \aleq \|f\|_{L^\infty(\R^d)}^{1-\theta}\, \|\Dso f\|_{L^{2d}(\R^d)}^\theta.
\]
\end{corollary}
\begin{proof}
First we consider the case $\beta > 1/2$.

Fix $\sigma \in [0,1]$ such that $\frac{1}{p} =(1-\sigma)\frac{\beta}{2d} + \sigma\frac{2\beta-1}{2d}$, i.e. $\sigma = \frac{\beta p-2d}{p(1-\beta)}$. Combining H\"older's inequality, Sobolev inequality, \Cref{la:sobolev}, and Gagliardo-Nirenberg inequality, \Cref{la:GagliardoNirenberg}, 
\[
\begin{split}
 \|\Ds{\beta} f\|_{L^{p}(\R^d)} =& \||\Ds{\beta} f|^{1-\sigma}\, |\Ds{\beta} f|^\sigma\|_{L^{p}(\R^d)}\\
\aleq& \|\Ds{\beta} f \|_{L^{\frac{2d}{\beta}}(\R^d)}^{1-\sigma}\ \|\Ds{\beta} f\|^\sigma_{L^{\frac{2d}{2\beta-1}}(\R^d)}\\
\aleq& \brac{\|f\|_{L^\infty(\R^d)}^{1-\beta} \|\Dso f \|_{L^{2d}(\R^d)}^\beta }^{1-\sigma}\ \|\Dso f\|^\sigma_{L^{2d}(\R^d)}\\
=&\|f\|_{L^\infty(\R^d)}^{(1-\beta)(1-\sigma)}\,  \|\Dso f\|_{L^{2d}(\R^d)}^{\beta (1-\sigma)+\sigma}\\
\end{split}
 \]
Thus, for 
\[
 \theta = \beta (1-\sigma)+\sigma = 2\brac{\beta-  \frac{d}{p}}
 \]
we conclude the case $\beta > \frac{1}{2}$.

Assume now $\beta \in (0,\frac{1}{2}]$ and $p \in [\frac{2d}{\beta},\infty)$. Pick $\gamma \in (\frac{1}{2},1]$ such that $\frac{\beta}{\gamma} p \leq \frac{2d}{2\gamma -1}$. Then we have with previous estimate (with $\theta = \frac{\gamma}{\beta} 2\brac{\beta- \frac{d}{p}}$),
\[
\begin{split}
 \|\Ds{\beta} f\|_{L^{p}(\R^d)} \aleq& \|f\|_{L^{\infty}}^{1-\frac{\beta}{\gamma}}\, \brac{\|\Ds{\gamma} f\|_{L^{\frac{\beta}{\gamma} p}(\R^d)} }^{\frac{\beta}{\gamma}} \\
 \aleq&\|f\|_{L^{\infty}}^{1-\frac{\beta}{\gamma}}\, \brac{\|f\|_{L^\infty(\R^d)}^{1-\theta}\, \|\Dso f\|_{L^{2d}(\R^d)}^{\theta} }^{\frac{\beta}{\gamma}}\\
 =&\|f\|_{L^{\infty}}^{1- 2\brac{\beta- \frac{d}{p}}}\, \|\Dso f\|_{L^{2d}(\R^d)}^{2\brac{\beta- \frac{d}{p}}}.
\end{split}
 \]
 \end{proof}

\subsection{Leibniz rule commutators}
In the following we discuss mostly Leibniz rule type estimates. 
% the commutator notation,
% \[
%  [T,a](b) := T(ab)-aTb.
% \]
The Leibniz rule operator for $\Ds{s}$ will be denoted by
\[
 H_{\Ds{s}}(f,g) := \Ds{s}(fg) - f \Ds{s} g - (\Ds{s} f) g.
\]
As a standing assumption, we are going to assume that all functions belong to $C_c^\infty(\R^d)$. By density arguments we can apply these inequalities to the situation in the next section. Let us stress that we make no effort to obtain the sharpest possible result with respect to $L^p$-spaces (in particular we generally rule out $p=1$ and $p=\infty$) but instead focus on the applicability for our purposes.

By a direct computation we have the following useful formula, which has been observed by many authors.
\begin{lemma}\label{la:Hdssformula}
Let $s \in (0,2)$ then for some $c = c(s,n)$,
\[
 H_{\Ds{s}}(f,g)(x) = c\int_{\R^d} \frac{(f(x)-f(y))(g(x)-g(y))}{|x-y|^{d+s}}\, dy
\]
\end{lemma}
% \begin{proof}
% Follows \[
% \begin{split}
%  &(a(x+h) b(x+h)+ a(x-h)b(x-h) -2a(x)b(x)) - (a(x+h) + a(x-h) -2a(x))b(x) - a(x) (b(x+h)+ b(x-h) -2b(x))\\
% &(a(x+h) -a(x) ) (b(x+h) -b(x))
% + (a(x-h)- a(x))(b(x-h) -b(x)) \\
% \end{split}
%  \]
%
% \end{proof}
We now begin by stating several useful estimates for the Leibniz rule operator, most of them are probably known to some experts -- and all of them can be proven via standard methods.
\begin{lemma}
For $\sigma \in (0,\alpha)$, $\alpha \in (0,1]$ we have for any $p,p_1,p_2 \in (1,\infty)$ with $\frac{1}{p} = \frac{1}{p_1} + \frac{1}{p_2}$,
\begin{equation}\label{eq:comm:5}
  \|H_{\Ds{\alpha}}(f, g)\|_{L^{p}(\R^d)} \aleq \|\Ds{\sigma} f\|_{L^{p_1}(\R^d)}\, \|\Ds{\alpha-\sigma} g\|_{L^{p_2}(\R^d)}.
\end{equation}
\end{lemma}
For a proof of \eqref{eq:comm:5} see e.g. \cite{LScomm}, or \cite[Theorem 3.4.1]{Ingmanns20}.

We can also estimate a differentiated version of the Leibniz rule operator.
\begin{lemma}\label{la:commie2}
Let $\alpha \in (0,1)$ and $\beta \in (0,1)$. Pick any $\gamma \in (0,1)$ such that $\alpha+\beta-\gamma \in (0,1)$, and $p,p_1,p_2 \in (1,\infty)$ such that 
 \[
  \frac{1}{p} = \frac{1}{p_1} + \frac{1}{p_2}.
 \]
Then
  \begin{equation}\label{eq:comm:123}
  \|\Ds{\alpha} H_{\Ds{\beta}} (f,g)\|_{L^{p}(\R^d)} \aleq \|\Ds{\gamma_1} f\|_{L^{p_1}}\, \|\Ds{\alpha +\beta-\gamma_1} g\|_{L^{p_2}}.
 \end{equation}
 
\end{lemma}
\begin{proof}[Proof of \eqref{eq:comm:123}]
This can be proven with techniques from \cite[Theorem 7.1.]{LScomm}, see also the presentation in \cite[Theorem 3.4.1]{Ingmanns20}. By duality we have 
\[
 \|\Ds{\alpha} H_{\Ds{\beta}} (f,g)\|_{L^{p}(\R^d)} \aleq \int_{\R^d} H_{\Ds{\beta}} (f,g)\, \Ds{\alpha} h
\]
for some $h \in C_c^\infty(\R^d)$  with $\|h\|_{L^{p'}(\R^d)} \leq 1$. Set $\tilde{h} := \Ds{\alpha} h$ and, as in \cite{LScomm,Ingmanns20}, let $F$, $G$, $H$ be the $\beta$-harmonic extension of $f$, $g$, $\tilde{h}$, respectively. That is, cf. \cite[(3.1.1)]{Ingmanns20},
\[
\begin{cases}
\div_{\R^{d+1}} (t^{1-s} \nabla_{\R^{d+1}} F) = 0 \quad &\text{for $(x,t) \in \R^{d+1}_+$}\\
\lim_{t \to 0^+}  F(x,t) = f(x)\quad &\text{for $x \in \R^{d}$}\\
\lim_{t \to \infty}  F(x,t) = 0\quad &\text{for $x \in \R^{d}$}\\
 \end{cases}
\]
An integration by parts argument, see \cite[Proof of Theorem 3.4.1]{Ingmanns20}, implies that
\[
 \abs{\int_{\R^d} H_{\Ds{\beta}} (f,g)\, h} \aleq \int_{\R^{d+1}_+} t^{1-\beta}\, |\nabla F|\, |\nabla G|\, |H|.
\]
Observe that by \cite[(10.6)]{LScomm}, we have the estimate
\[
 |H(x,t)| \leq t^{-\alpha} \mathcal{M} (\lapms{\alpha} \tilde{h})(x) = t^{-\alpha} \mathcal{M} h(x),
\]
where $\mathcal{M}$ is the Hardy-Littlewood maximal function. That is, we have found
\[
\begin{split}
 &\abs{\int_{\R^d} H_{\Ds{\beta}} (f,g)\, \Ds{\alpha} h} \\
 \aleq&\int_{\R^{d}} \mathcal{M} h(x)  \int_{t=0}^\infty t^{1-\beta-\alpha} |\nabla F(x,t)|\, |\nabla G(x,t)|\, dt\, \, dx\\
 \aleq&\|\mathcal{M} h\|_{L^{p'}(\R^d)}\, \norm{\brac{\int_{0}^\infty \brac{t^{\frac{1}{2}-\gamma} |\nabla F(x,t)|}^2 dt}^{\frac{1}{2}} }_{L^{p_1}(\R^d)}\, \norm{\brac{\int_{t=0}^\infty \brac{t^{\frac{1}{2}-(\beta+\alpha-\gamma)} |\nabla G(x,t)|}^2\, dt}^{\frac{1}{2}}}_{L^{p_2}(\R^d)} \\
 \aleq&\|h\|_{L^{p'}(\R^d)}\, [f]_{\dot{F}^\gamma_{p_1,2}(\R^d)}\, [g]_{\dot{F}^{\beta + \alpha -\gamma}_{p_2,2}(\R^d)}\\
 \aeq&\|h\|_{L^{p'}(\R^d)}\, \|\Ds{\gamma} f\|_{L^{p_1}(\R^d)}\, \|\Ds{\beta+\alpha-\gamma}g\|_{L^{p_2}(\R^d)}
 \end{split}
\]
In the second to last inequality we used the boundedness of the maximal function $\mathcal{M}$ on $L^{p'}$ and the identification of Triebel spaces, \cite[Theorem 10.8]{LScomm}. In the last step we used that by Littlewood-Paley-theorem, \cite[Theorem 1.3.8.]{GrafakosModern}, $[f]_{\dot{F}^\gamma_{p_1,2}} \aeq \|\Ds{\gamma} f\|_{L^{p_1}(\R^d)}$. We can conclude.
% 
% By \Cref{la:Hdssformula} we have
% \[
% \begin{split}
%  H_{\beta} (f,g)(x) =& c\int_{\R^d} \frac{(f(x)-f(y))(g(x)-g(y))}{|x-y|^{n+\beta}}\, dy\\
% =& c\int_{\R^d} \frac{(f(x)-f(x+h))(g(x)-g(x+h))}{|h|^{n+\beta}}\, dh\\
%  \end{split}
%  \]
% Thus,
% \[
% \begin{split}
%  \Ds{\alpha} H_{\beta} (f,g)(x)=& c\int_{\R^d} \frac{(\Ds{\alpha}f(x)-\Ds{\alpha}f(x+h))(g(x)-g(x+h))}{|h|^{n+\beta}}\, dh\\
%  &+ c\int_{\R^d} \frac{(f(x)-f(x+h))(\Ds{\alpha}g(x)-\Ds{\alpha} g(x+h))}{|h|^{n+\beta}}\, dh\\
%  &+c\int_{\R^d} \frac{H_{\Ds{\alpha}} (f(\cdot)-f(\cdot+h), g(\cdot)-g(\cdot+h))(x)}{|h|^{n+\beta}}\, dh\\
%   \end{split}
%  \]
%  \ToDo first two terms
% \[
%  \begin{split}
%  &\int_{\R^d} \frac{H_{\Ds{\alpha}} (f(\cdot)-f(\cdot+h), g(\cdot)-g(\cdot+h))(x)}{|h|^{n+\beta}}\, dh\\
%   =&\int_{\R^d} \int_{\R^d} \frac{(f(x)-f(x+h) - (f(y)-f(y+h))) (g(x)-g(x+h) - (g(y)-g(y+h)))}{|x-y|^{n+\alpha} |h|^{n+\beta}}\, dx dh\\
%  \end{split}
% \]
% 
% 
% 
% For any $\xi,\eta \in \R^n$ we have 
% \[
%  \abs{\abs{\xi}^{\beta}-\abs{\xi-\eta}^{\beta}-\abs{\eta}^{\beta}} \aleq \min\{|\xi-\eta|^{\beta},\abs{\eta}^{\beta}\}
% \]
% So
% \[
% \begin{split}
% & |\xi|^\alpha \abs{\abs{\xi}^{\beta}-\abs{\xi-\eta}^{\beta}-\abs{\eta}^{\beta}}\\
% \aleq& |\xi|^\alpha \min\{|\xi-\eta|^{\beta},\abs{\eta}^{\beta}\}\\
%  \aleq& |\xi-\eta|^\alpha \min\{|\xi-\eta|^{\beta},\abs{\eta}^{\beta}\}\\
%  &+ |\eta|^\alpha \min\{|\xi-\eta|^{\beta},\abs{\eta}^{\beta}\}\\
%  \aleq& |\xi-\eta|^{\gamma_1} \abs{\eta}^{\beta+\alpha-\gamma_1}\\
%  &+ |\eta|^{\gamma_2} |\xi-\eta|^{\beta+\alpha-\gamma_2}
%  \end{split}
% \]
% \ToDo details!
\end{proof}

We will also need an estimate for a double commutator. We are not aware of this estimate in the literature, but it can be obtained with the usual paraproduct approach.
\begin{lemma}\label{la:doublecommie}
Let $\alpha, \beta \in (0,1]$, and consider the double commutator
 \[
 \tilde{H}_{\Ds{\beta},\Ds{\alpha}}(f,g) := \Ds{\beta} H_{\Ds{\alpha}}(f,g) - H_{\Ds{\alpha}} (\Ds{\beta} f, g) - H_{\Ds{\alpha}} (f, \Ds{\beta} g).
\]
Then for any $\gamma \in (0,\alpha + \beta)$ and any $p,p_1,p_2 \in (1,\infty)$ with $\frac{1}{p} = \frac{1}{p_1} + \frac{1}{p_2}$ we have
\[
 \|\tilde{H}_{\Ds{\beta},\Ds{\alpha}}(f,g)\|_{L^{p}(\R^d)} \aleq \|\Ds{\gamma}  f\|_{L^{p_1}(\R^d)}\, \|\Ds{\alpha+\beta-\gamma}  f\|_{L^{p_2}(\R^d)}.
\]
\end{lemma}
\begin{proof}
This can be proven in a very similar fashion to \cite[Section 3]{DLR2011}, we only sketch the main steps.

Denote by $\dot{F}^{s}_{p,q}$ the homogeneous Triebel spaces, and by $\Delta_k$ the Littlewood-Paley projection operator, \cite[1.3.2]{GrafakosModern}. Then by Littlewood-Paley theorem, \cite[Theorem 1.3.8.]{GrafakosModern}, and duality,
for some $\psi\in C_c^\infty(\R^d)$ with $[\psi]_{\dot{F}^{0}_{p',2}}(\R^d) \aleq 1$,
 \[
\begin{split}
 \norm{\tilde{H}_{\Ds{\beta},\Ds{\alpha}}(f,g)}_{L^{p}(\R^d)} \aeq& \norm{\tilde{H}_{\Ds{\beta},\Ds{\alpha}}(f,g)}_{\dot{F}^{0}_{p,2}(\R^d)}\\
 \aeq& \sum_{k \in \Z} \int_{\R^d} \tilde{H}_{\Ds{\beta},\Ds{\alpha}}(f,g)\, \Delta_k \psi.
 \end{split}
\]
Here $\Delta_k$ denotes the Littlewood-Paley projection operator onto the $2^k$-frequency.
With the usual paraproduct argument, denoting by $\Delta^\ell := \sum_{\tilde{\ell} \leq \ell}\Delta_{\tilde{\ell}}$,
 \begin{equation}\label{eq:paraprod:2345}
\begin{split}
 \|\tilde{H}_{\Ds{\beta},\Ds{\alpha}}(f,g)\|_{L^{p_1}(\R^d)} \aeq& 
 \sum_{j \in \Z} \sum_{k \aeq j} \int_{\R^d} \tilde{H}_{\Ds{\beta},\Ds{\alpha}}(\Delta_j f,\Delta^{j-4}g)\, \Delta_k \psi\\
 &+\sum_{j \in \Z} \sum_{k \aeq j} \int_{\R^d} \tilde{H}_{\Ds{\beta},\Ds{\alpha}}(\Delta^{j-4} f,\Delta_{j}g)\, \Delta_k \psi\\
 &+\sum_{j \in \Z} \sum_{\ell \aeq j} \int_{\R^d} \tilde{H}_{\Ds{\beta},\Ds{\alpha}}(\Delta_{j} f,\Delta_{\ell}g)\, \Delta^{j-10} \psi\\
  &+\sum_{j \in \Z} \sum_{\ell \aeq k \aeq  j} \int_{\R^d} \tilde{H}_{\Ds{\beta},\Ds{\alpha}}(\Delta_{j} f,\Delta_{\ell}g)\, \Delta_k \psi.
 \end{split}
\end{equation}
Here, by a slight abuse of notation we say for indices $\ell, k$ that $\ell \aeq k$ if $k-c\leq \ell \leq k + c$ for some constant $c > 0$.

The last term in \eqref{eq:paraprod:2345} is the simplest to estimate, since for any $\theta_i \in [0,\alpha +\beta]$ such that $\sum_{i=1}^3 \theta_i = \alpha+\beta$
\[
\begin{split}
 &\sum_{j \in \Z} \sum_{\ell \aeq k \aeq  j} \int_{\R^d} \Ds{\theta_1}\Delta_{j} f\, \Ds{\theta_2}\Delta_{\ell}g\, \Ds{\theta_3}\Delta_k \psi\\
 \aeq&\sum_{j \in \Z} \sum_{\ell \aeq k \aeq  j} \int_{\R^d} 2^{(\gamma-\theta_1)j}\Ds{\theta_1}\Delta_{j} f\ 2^{(\alpha+\beta-\gamma-\theta_2)\ell}\Ds{\theta_2}\Delta_{\ell}g\ 2^{-\theta_3 k}\Ds{\theta_3}\Delta_k \psi\\
 \aleq& \norm{\brac{\sum_{j} |2^{(\gamma-\theta_1)j}\Ds{\theta_1}\Delta_{j} f|^{4}}^{\frac{1}{4}}}_{L^{p_1}(\R^d)}\, \norm{\brac{\sum_{\ell \in \Z} \brac{2^{(\alpha+\beta-\gamma-\theta_2)\ell} \Ds{\theta_2}\Delta_{\ell}g}^4}^{\frac{1}{4}}}_{L^{p_2}(\R^d)}\\
 & \quad \cdot \norm{\brac{\sum_{k\in \Z} \brac{2^{-\theta_3 k}\Ds{\theta_3}\Delta_k \psi}^2}^{\frac{1}{2}}}_{L^{p'}(\R^d)}\\
 \aleq&[f]_{\dot{F}^{\gamma}_{p_1,4}(\R^d)}\, [g]_{\dot{F}^{\alpha+\beta-\gamma}_{p_2,4}(\R^d)}\, [\psi]_{\dot{F}^0_{p',2}(\R^d)}\\
 \aleq&[f]_{\dot{F}^{\gamma}_{p_1,2}(\R^d)}\, [g]_{\dot{F}^{\alpha+\beta-\gamma}_{p_2,2}(\R^d)}\, [\psi]_{\dot{F}^0_{p',2}(\R^d)}\\
 \aeq&\|\Ds{\gamma} f\|_{L^{p_1}(\R^d)}\, \|\Ds{\alpha+\beta-\gamma} g\|_{L^{p_3}(\R^d)}.
 \end{split}
\]
The other terms are very similar to each other. We only discuss the first term.  By Plancherel theorem we have 
\[
\begin{split}
&\int_{\R^d} \tilde{H}_{\Ds{\beta},\Ds{\alpha}}(\Delta_j f,\Delta^{j-4}g)\, \Delta_k \psi\\
=&c\int_{\R^d}\int_{\R^d} k(\xi,\eta)\, \widehat{\Delta_j f}(\eta)\, \widehat{\Delta^{j-4}g}(\xi-\eta)\, \widehat{\Delta_k \psi}(\xi)\\
\end{split}
\]
Here $k$ is the symbol of the operator $\tilde{H}_{\Ds{\beta},\Ds{\alpha}}$ given by
 \[
 \begin{split}
  k(\xi,\eta) :=&|\xi|^\beta \brac{|\xi|^\alpha - |\eta|^\alpha - |\xi-\eta|^\alpha}\\
  &- |\eta|^\beta \brac{|\xi|^\alpha - |\eta|^\alpha - |\xi-\eta|^\alpha}\\
  &- |\xi-\eta|^\beta\brac{|\xi|^\alpha - |\eta|^\alpha - |\xi-\eta|^\alpha}\\
  =&\brac{|\xi|^\beta - |\eta|^\beta - |\xi-\eta|^\beta}\, \brac{|\xi|^\alpha - |\eta|^\alpha - |\xi-\eta|^\alpha}\\
  \end{split}
\]
We observe that by the support of the Littlewood-Paley projection operators $\Delta_j$ and $\Delta^{j-4}$, 
in the integral above we have $|\xi-\eta| \leq \frac{1}{2} |\eta|$. By a Taylor expansion,
\[
\begin{split}
  k(\xi,\eta) = |\xi-\eta|^{\alpha+\beta} \brac{1+ \sum_{\ell=1}^\infty \frac{1}{\ell!} m_\ell(\eta) n_\ell(\xi-\eta)\, |\xi-\eta|^{\ell} |\eta|^{-\ell} }
  \end{split}
\]
where $m_\ell(\eta)$ and $n_\ell(\eta)$ are zero homogeneous functions. Now we observe that for any $\theta \geq 0$
\[
\begin{split}
 &\sum_{j \in \Z} \sum_{k \aeq j}\int_{\R^d}\int_{\R^d} |\xi-\eta|^{\alpha + \beta + \theta}\, |\eta|^{-\theta}\, \widehat{\Delta_j f}(\eta)\, \widehat{\Delta^{j-4}g}(\xi-\eta)\, \widehat{\Delta_k \psi}(\xi)\\
 =&c\, \sum_{j \in \Z} \sum_{k \aeq j}\int_{\R^d} \lapms{\theta} \Delta_j f\ \Ds{\alpha+\beta+\theta} \Delta^{j-4} g\, \Delta_k \psi\\
 \aleq&\norm{ \brac{\sum_{j \in \Z} \brac{2^{j(\gamma +\theta)}\lapms{\theta} \Delta_j f}^2}^{\frac{1}{2}} }_{L^{p_1}(\R^d)} \ \norm{\sup_{j} 2^{-j(\gamma+\theta)}\Ds{\alpha+\beta+\theta} \Delta^{j-4} g}_{L^{p_2}(\R^d)}\, \norm{\brac{\sum_{k \in \Z} \brac{\Delta_k \psi}^2}^{\frac{1}{2}}}_{L^{p'}(\R^d)}\\
 \aeq&[f]_{\dot{F}^{\gamma}_{p_1,2}(\R^d)} \ \norm{\sup_{j} 2^{-j(\gamma+\theta)}\Ds{\alpha+\beta+\theta} \Delta^{j-4} g}_{L^{p_2}(\R^d)}\, \|\psi\|_{L^{p'}(\R^d)}.
 \end{split}
\]
In the last inequality we used the Littlewood-Paley theorem, \cite[Theorem 1.3.8.]{GrafakosModern}.
Since $\gamma + \theta > 0$
\[
\begin{split}
 &2^{-j(\gamma+\theta)}\Ds{\alpha+\beta+\theta} \Delta^{j-4} g \\
 \leq &\sum_{\ell \leq j-4} 2^{-\ell(\gamma+\theta)}\Ds{\alpha+\beta+\theta} \Delta_{\ell} g\, 2^{(\ell-j)(\gamma+\theta)} \\
 \aleq&\brac{\sum_{\ell \leq j-4} \brac{2^{-\ell(\gamma+\theta)} \Ds{\alpha+\beta+\theta} \Delta_{\ell} g}^2}^{\frac{1}{2}} \\
 \end{split}
\]
and thus we actually have
\[
\begin{split}
 &\sum_{j \in \Z} \sum_{k \aeq j}\int_{\R^d}\int_{\R^d} |\xi-\eta|^{\alpha + \beta + \theta}\, |\eta|^{-\theta}\, \widehat{\Delta_j f}(\eta)\, \widehat{\Delta^{j-4}g}(\xi-\eta)\, \widehat{\Delta_k \psi}(\xi)\\
 \aleq&\norm{f}_{\dot{F}^{\gamma}_{p_1,2}(\R^d)} \ \norm{g}_{\dot{F}^{\alpha+\beta-\gamma}_{p_2,2}(\R^d)}\, \norm{\psi}_{\dot{F}^{0}_{p',2}(\R^d)}\\
 \aleq&\|\Ds{\gamma} f\|_{L^{p_1}(\R^d)}\, \|\Ds{\alpha+\beta-\gamma} g\|_{L^{p_2}(\R^d)}.
 \end{split}
\]
With this method we can estimate each of the terms in \eqref{eq:paraprod:2345} and obtain the claim.

\end{proof}

The next result is very similar to the estimate of \Cref{la:commie2} (which indeed can be proven with the techniques of the following lemma).
\begin{lemma}\label{la:commies3toy} 
Assume $s \in (0,1]$ and $\alpha_1,\alpha_2 \in (0,s)$ such that $\sum_{i=1}^2 \alpha_i = s$. Let $p \in (1,\infty)$. Assume for $p_1,p_2 \in [2,\infty)$ such that 
\begin{equation} \label{eq:commies3toy:pq} 
\frac{1}{p_1}+\frac{1}{p_2}  = \frac{1}{p}. 
\end{equation}
Then we have
\[
\begin{split}
 &\brac{\int_{\R^d} \brac{\int_{\R^d} \frac{\abs{f(x)-f(y)}\, \abs{g(x)-g(y)}} {|x-y|^{d+s}}dy}^p dx}^{\frac{1}{p}}\\
 \aleq &\|\Ds{\alpha_1} f\|_{L^{p_1}(\R^d)}\, \|\Ds{\alpha_2} g\|_{L^{p_2}(\R^d)}\,\\
 \end{split}
\]
\end{lemma}
\begin{proof}
From H\"older's inequality
\[
\begin{split}
 &\brac{\int_{\R^d} \brac{\int_{\R^d} \frac{\abs{f(x)-f(y)}\, \abs{g(x)-g(y)}} {|x-y|^{d+s}}dy}^p dx}^{\frac{1}{p}}\\
 \aleq &\brac{\int_{\R^d} \brac{\int_{\R^d} \frac{\abs{f(x)-f(y)}^{2}}{|x-y|^{d+2\alpha_1 }} dy}^{p_1} dx}^{\frac{1}{p_1}}\
  \brac{\int_{\R^d} \brac{\int_{\R^d} \frac{\abs{g(x)-g(y)}^{2}}{|x-y|^{d+2\alpha_2}} dy}^{p_2} dx}^{\frac{1}{p_2}}\\
 \end{split}
\]
Since $p_i \geq 2$ we have by the results in \cite{P17}, denoting by $\dot{F}$ the homogeneous Triebel-Lizorkin space, cf. \cite{RS96}, and Littlewood-Paley theorem, \cite[Theorem 1.3.8.]{GrafakosModern},
\[
 \brac{\int_{\R^d} \brac{\int_{\R^d} \frac{\abs{h(x)-h(y)}^{2}}{|x-y|^{d+2\alpha_i }} dy}^{p_i} dx}^{\frac{1}{p_1}} \aeq [h]_{\dot{F}^{\alpha_i,p_i}_{2}(\R^d)} \aeq \|\Ds{\alpha_i} h\|_{L^{p_i}(\R^d)}.
\]
We can conclude.
\end{proof}

We will need Leibniz-rule estimates involving three terms, the basis of which is the following Lemma.
\begin{lemma}\label{la:commies3} 
Assume $\alpha_i \in (0,1)$ such that $\sum_{i=1}^3 \alpha_i = 1$ and let $p \in (1,\infty)$. Assume for $p_i \in (p,\infty)$ such that 
\[\frac{1}{p_1}+\frac{1}{p_2} + \frac{1}{p_3} = \frac{1}{p} ,\]
and 
\begin{equation}\label{eq:c3:23525}
 \frac{1}{p_i} - \frac{\alpha_i}{d} < \frac{1}{2}, \quad i=1,2,3.
\end{equation}
(Observe the previous assumptions are trivially satisfied if $p_i \geq 2$).

Then we have 
\[
\begin{split}
 &\brac{\int_{\R^d} \brac{\int_{\R^d} \frac{\abs{f(x)-f(y)}\, \abs{g(x)-g(y)} \abs{h(x)-h(y)}} {|x-y|^{d+1}}dy}^p dx}^{\frac{1}{p}}\\
 \aleq &\|\Ds{\alpha_1} f\|_{L^{p_1}(\R^d)}\, \|\Ds{\alpha_2} g\|_{L^{p_2}(\R^d)}\, \|\Ds{\alpha_3} h\|_{L^{p_3}(\R^d)}.
 \end{split}
\]
\end{lemma}
\begin{proof}
Since 
\[\sum_{i=1}^3 \frac{d-\alpha_i p_i}{dp_i} < \frac{1}{p_1} + \frac{1}{p_2} + \frac{1}{p_3} = \frac{1}{p} < 1\] 
in view of \eqref{eq:c3:23525} we can find $q_i \in [2,\frac{dp_i}{d-\alpha_i p_i})$ (if $p_i > \frac{d}{\alpha_i}$ we pick $q_i \in [2,\infty)$) such that 
\[
 \frac{1}{q_1} + \frac{1}{q_2} + \frac{1}{q_3} = 1.
\]
Then
\[
\begin{split}
&\int_{\R^d} \frac{\abs{f(x)-f(y)}\, \abs{g(x)-g(y)} \abs{h(x)-h(y)}} {|x-y|^{d+1}}dy\\
\aleq& \brac{\int_{\R^d} \frac{\abs{f(x)-f(y)}^{q_1}}{|x-y|^{d+\alpha_1 q_1}} dy}^{\frac{1}{q_1}}\, \brac{\int_{\R^d} \frac{\abs{g(x)-g(y)}^{q_2}}{|x-y|^{d+\alpha_2 q_2}} dy}^{\frac{1}{q_2}}\, \brac{\int_{\R^d} \frac{\abs{h(x)-h(y)}^{q_3}}{|x-y|^{d+\alpha_3 q_3}} dy}^{\frac{1}{q_3}}
\end{split}
 \]
From another application of H\"older's inequality,
\[
\begin{split}
 &\brac{\int_{\R^d} \brac{\int_{\R^d} \frac{\abs{f(x)-f(y)}\, \abs{g(x)-g(y)} \abs{h(x)-h(y)}} {|x-y|^{d+1}}dy}^p dx}^{\frac{1}{p}}\\
 \aleq&[f]_{W^{\alpha_1,p_1}_{q_1}(\R^d)}\, [g]_{W^{\alpha_2,p_2}_{q_2}(\R^d)}\, [h]_{W^{\alpha_3,p_3}_{q_3}(\R^d)}.
 \end{split}
\]
Here the $W^{\alpha,p}_{q}$-seminorm for $\alpha \in (0,1)$ and $p,q \in (1,\infty)$ is defined as
\[
 [f]_{W^{\alpha,p}_q(\R^d)} = \brac{\int_{\R^d} \brac{\int_{\R^d} \frac{|f(x)-f(y)|^q}{|x-y|^{d+sq}} dy}^{\frac{p}{q}}\, dx}^{\frac{1}{p}}.
\]
Our choice for $q_i$ ensures $p_i >\frac{dq_i}{d+\alpha_i q_i}$, so we have by the results in \cite{P17},
\[
 [f]_{W^{\alpha_i,p_i}_{q_i}(\R^d)} \aeq [f]_{\dot{F}^{\alpha_i,p_i}_{q_i}(\R^d)},
\]
where $\dot{F}$ denotes the homogeneous Triebel-Lizorkin space. That is, we have 
\[
\begin{split}
 &\brac{\int_{\R^d} \brac{\int_{\R^d} \frac{\abs{f(x)-f(y)}\, \abs{g(x)-g(y)} \abs{h(y)-h(y)}} {|x-y|^{d+1}}dy}^p dx}^{\frac{1}{p}}\\
 \aleq&[f]_{\dot{F}^{\alpha_1,p_1}_{q_1}(\R^d)}\, [g]_{\dot{F}^{\alpha_2,p_2}_{q_2}(\R^d)}\, [h]_{\dot{F}^{\alpha_3,p_3}_{q_3}(\R^d)}.
 \end{split}
\]
Since $q_i \geq 2$ we have, cf. \cite{RS96},
\[
 [f]_{\dot{F}^{\alpha_i,p_i}_{q_i}(\R^d)} \aleq [f]_{\dot{F}^{\alpha_i,p_i}_{2}(\R^d)} \aeq \|\Ds{\alpha_i} f\|_{L^{p_i}(\R^d)}.
\]
Thus, we have established the claim and can conclude.
% \[
% \begin{split}
%  &\brac{\int_{\R^d} \brac{\int_{\R^d} \frac{\abs{f(x)-f(y)}\, \abs{g(x)-g(y)} \abs{h(y)-h(y)}} {|x-y|^{d+1}}dy}^p dx}^{\frac{1}{p}}\\
%  \aleq&
% %  &[f]_{\dot{F}^{\alpha_1,p_1}_{2}(\R^d)}\, [g]_{\dot{F}^{\alpha_2,p_2}_{2}(\R^d)}\, [h]_{\dot{F}^{\alpha_3,p_3}_{2}(\R^d)}\\
% %  \aeq&
%  \|\Ds{\alpha_1}f\|_{L^{p_1}(\R^d)}\, \|\Ds{\alpha_2}f\|_{L^{p_2}(\R^d)}\, \|\Ds{\alpha_3}f\|_{L^{p_3}(\R^d)}\\
%  \end{split}
% \]
\end{proof}

We now state a version similar to \Cref{la:commies3} but for $\alpha_3 < 0$.
\begin{lemma}\label{la:commies3b} 
Assume $\alpha_i \in (0,1)$, $i=1,2$ such that $\alpha_1 + \alpha_2 > 1$. If $d=1$ assume moreover that $\alpha_1+\alpha_2-1 < \frac{d}{2}$. Assume for $p_i \in (2,\infty)$ such that $\frac{1}{p_1}+\frac{1}{p_2} + \frac{1}{p_3} = \frac{1}{2}$ and $\frac{dp_3}{ d-(1-\alpha_1-\alpha_2) p_3} \in (1,\infty)$. Then we have 
\[
\begin{split}
 &\brac{\int_{\R^d} \brac{\int_{\R^d} \frac{\abs{f(x)-f(y)}\, \abs{g(x)-g(y)} \abs{h(y)}} {|x-y|^{d+1}}dy}^2 dx}^{\frac{1}{2}}\\
 \aleq &\|\Ds{\alpha_1} f\|_{L^{p_1}(\R^d)}\, \|\Ds{\alpha_2} g\|_{L^{p_2}(\R^d)}\, \|h\|_{L^{\frac{dp_3}{d+(\alpha_1+\alpha_2-1) p_3}}(\R^d)}.
 \end{split}
\]
\end{lemma}
\begin{proof}
As in \Cref{la:commies3}, since $p_i > 2$ for $i=1,2$ we can choose $q_1 = p_1$ and $q_2 = p_2$ and set
\[
 q_3 = \frac{2p_3}{p_3+2} \in (p_3,\infty).
\]
Set $-\alpha_3 := 1 - \alpha_1 - \alpha_2 < 0$. Then $\alpha_3 \in (0,1)$. By the same argument as in the proof of \Cref{la:commies3},
 \[
\begin{split}
&\int_{\R^d} \frac{\abs{f(x)-f(y)}\, \abs{g(x)-g(y)} \abs{h(y)}} {|x-y|^{d+1}}dy\\
\aleq& \brac{\int_{\R^d} \frac{\abs{f(x)-f(y)}^{q_1}}{|x-y|^{d+\alpha_1 q_1}} dy}^{\frac{1}{q_1}}\, \brac{\int_{\R^d} \frac{\abs{g(x)-g(y)}^{q_2}}{|x-y|^{d+\alpha_2 q_2}} dy}^{\frac{1}{q_2}}\, \brac{\int_{\R^d} \frac{\abs{h(y)}^{q_3}}{|x-y|^{d-\alpha_3 q_3}} dy}^{\frac{1}{q_3}}\\
\aeq& \brac{\int_{\R^d} \frac{\abs{f(x)-f(y)}^{q_1}}{|x-y|^{d+\alpha_1 q_1}} dy}^{\frac{1}{q_1}}\, \brac{\int_{\R^d} \frac{\abs{g(x)-g(y)}^{q_2}}{|x-y|^{d+\alpha_2 q_2}} dy}^{\frac{1}{q_2}}\, \brac{ \lapms{\alpha_3 q_3} \brac{\abs{h}^{q_3}} (x)}^{\frac{1}{q_3}}.
\end{split}
 \]
So, with the same H\"older inequality and embedding theorems as in \Cref{la:commies3},
\[
\begin{split}
 &\brac{\int_{\R^d} \brac{\int_{\R^d} \frac{\abs{f(x)-f(y)}\, \abs{g(x)-g(y)} \abs{h(y)-h(y)}} {|x-y|^{d+1}}dy}^2 dx}^{\frac{1}{2}}\\
 \aleq&\|\Ds{\alpha_1} f\|_{L^{p_1}(\R^d)}\, \|\Ds{\alpha_2} g\|_{L^{p_1}(\R^d)}\,   \|\lapms{\alpha_3 q_3}\brac{\abs{h}^{q_3}}^{\frac{1}{q_3}}\|_{L^{p_3}(\R^d)}.
\end{split}
\]
We now observe that by Sobolev inequality, \Cref{la:sobolev},
\[
\begin{split}
 \|\lapms{\alpha_3 q_3}\brac{\abs{h}^{q_3}}^{\frac{1}{q_3}}\|_{L^{p_3}(\R^d)} =& \|\lapms{\alpha_3 q_3}\brac{\abs{h}^{q_3}}\|_{L^{\frac{p_3}{q_3}}(\R^d)}^{\frac{1}{q_3}}\\
 \aleq& \|\abs{h}^{q_3}\|_{L^{\frac{d p_3}{q_3 d+\alpha_3 p_3q_3}}(\R^d)}^{\frac{1}{q_3}}\\
 =&\|h\|_{L^{\frac{dp_3}{ d+\alpha_3 p_3}}(\R^d)}.
 \end{split}
\]
The above is correct as long as $\frac{d p_3}{q_3 d+\alpha_3 p_3q_3} \in (1,\infty)$. 
% To ensure this, we observe 
% \[
%  \frac{d p_3}{q_3 d+\alpha_3 p_3q_3} = \frac{p_3}{q_3}  \frac{d }{d+\alpha_3 p_3} = \frac{p_3+2}{\frac{2\alpha_3}{d} p_3+2}.
% \]
We see that if $\alpha_3 < \frac{d}{2}$ this is satisfied and can conclude.
\end{proof}

\subsection{Specific estimates}
In this section we record estimates for specific $L^p$-spaces of interest. These are mostly consequences from the estimates above, and will be useful throughout the next section. 

\begin{lemma}
Let $d \geq 2$ then
  \begin{equation}\label{eq:comm:4}
 \|H_{\Dso{}} (f,g)\|_{L^{\frac{2d}{d-1}}(\R^d)} \aleq \|\Dso{} f\|_{L^{2}(\R^d)}\, \|\Dso{} g\|_{L^{2d}(\R^d)}.
 \end{equation}
\begin{equation}\label{eq:comm:1}
  \|\Dso{} H_{\Dso{}} (f,g)\|_{L^{d}(\R^d)} \aleq \|\Dso{} f\|_{L^{2d}(\R^d)}\, \|\Dso{} g\|_{L^{2d}(\R^d)}
 \end{equation}
\begin{equation}\label{eq:comm:3}
  \|H_{\Dso{}} (f,g)\|_{L^{\infty}(\R^d)} \aleq \|\Dso{} f\|_{L^{2d,2}(\R^d)}\, \|\Dso{} g\|_{L^{2d,2}(\R^d)}
 \end{equation}
\end{lemma}

\begin{proof}[Proof of \eqref{eq:comm:4}]
From \eqref{eq:comm:5} we have 
\[
\|H_{\Dso{}} (f,g)\|_{L^{\frac{2d}{d-1}}(\R^d)} \aleq \||\Ds{1/3}f|\|_{L^{\frac{6d}{3d-4}}(\R^d)}\, \|\Ds{2/3} g\|_{L^{6d}(\R^d)}\\
\] 
Now the claim follows by Sobolev embedding, \Cref{la:sobolev}.
\end{proof}

\begin{proof}[Proof of \eqref{eq:comm:1} and \eqref{eq:comm:3}]
% 

% 
% Observe that
% \[
% \abs{\xi} \abs{\abs{\xi}-\abs{\xi-\eta}-\abs{\eta}} \aleq \abs{\xi-\eta}\, \abs{\eta}.
% \]
% Indeed,
% We have by triangular inequality
% \[
%  \abs{\abs{\xi}-\abs{\xi-\eta}-\abs{\eta}} \leq 2 \min\{|\xi-\eta|,\abs{\eta}\}
% \]
% So
% \[
%  \abs{\xi} \abs{\abs{\xi}-\abs{\xi-\eta}-\abs{\eta}}  \aleq \abs{\xi-\eta} \min\{|\xi-\eta|,\abs{\eta}\}+ \abs{\eta}\min\{|\xi-\eta|,\abs{\eta}\} \aleq \abs{\xi-\eta}\, \abs{\eta}.
% \]
% Rest should be convolution.
\eqref{eq:comm:1} can be proven almost without changes following the proof of \cite[Theorem 8.2.]{LS2018}, see also \cite[Theorem 3.5.2]{Ingmanns20}, where such an estimate was obtained for the $L^1$-case (even the Hardy-space). Alternatively, one could use paraproduct estimates as in \cite[Theorem 1.4.]{Seps15}.

Actually, from that argument (or by interpolation) we obtain an estimate in the realm of Lorentz spaces.
\[
  \|\Dso{} H_{\Dso{}} (f,g)\|_{L^{d,1}(\R^d)} \aleq \|\Dso{} f\|_{L^{2d,2}(\R^d)}\, \|\Dso{} g\|_{L^{2d,2}(\R^d)}.
 \]
The latter implies \eqref{eq:comm:3}, using the Sobolev embedding \Cref{la:sobolev}, \eqref{eq:limitsobLinfty}.
\end{proof}

For later use we also record the following easy consequence of the Leibniz rule estimate
\begin{lemma}\label{la:wugsd8gfus0df}
Assume $d\geq 3$, $\sigma \in [0,1]$. Then
\[
 \|\Ds{\sigma} \brac{f\, \Dso{}g}\|_{L^{\frac{2d}{d-1+2\sigma}}(\R^d)} \aleq \|\Dso f\|_{L^2(\R^d)}\, \|\Ds{1+\sigma} g\|_{L^{\frac{2d}{2(1+\sigma)-1}}(\R^d)}.
\]
\end{lemma}
\begin{proof}
With the help of Leibniz rules and Sobolev embedding, \Cref{la:sobolev},
% \begin{equation}\label{eq:wugsd8gfus0df}
\[
\begin{split}
 &\|\Ds{\sigma} \brac{f\, \Dso{}g}\|_{L^{\frac{2d}{d-1+2\sigma}}(\R^d)}\\
 \aleq&\|\abs{\Ds{\sigma} f}\, \abs{\Dso{}g}\|_{L^{\frac{2d}{d-1+2\sigma}}(\R^d)} + \|\abs{f}\, \abs{\Ds{1+\sigma}g}\|_{L^{\frac{2d}{d-1+2\sigma}}(\R^d)} + \|H_{\Ds{\sigma}}(f,\Dso g)\|_{L^{\frac{2d}{d-1+2\sigma}}(\R^d)}\\
 \overset{\eqref{eq:comm:5}}{\aleq}&\|\Ds{\sigma} f\|_{L^{\frac{2d}{d-2(1-\sigma)}}(\R^d)}\, \|\Dso{}g\|_{L^{2d}(\R^d)} + \|f\|_{L^{\frac{2d}{d-2}}(\R^d)} \|\Ds{1+\sigma}g\|_{L^{\frac{2d}{2(1+\sigma)-1}}(\R^d)} \\
  &+ \|\Ds{\frac{\sigma}{2}}f\|_{L^{\frac{2d}{d-2+\sigma}}(\R^d)} \|\Ds{1+\frac{\sigma}{2}} g\|_{L^{\frac{2d}{1+\sigma}}(\R^d)}\\
\aleq&\|\Dso f\|_{L^2(\R^d)}\, \|\Ds{1+\sigma} g\|_{L^{\frac{2d}{2(1+\sigma)-1}}(\R^d)}.
 \end{split}
\]
% \end{equation}
\end{proof}

Next we record an estimate for another version of a sort of double Leibniz rule.
\begin{lemma}\label{la:commie3cde} 
For $\alpha \in (\frac{1}{2},1)$,  $d \geq 2$
\[
 \|H_{\Dso{}}(f h,g) - f  H_{\Dso}(h,g)\|_{L^2(\R^d)} \aleq \|\Ds{\alpha} f\|_{L^{\frac{2d}{2\alpha-1}}(\R^d)}\, \|\Ds{\alpha} g\|_{L^{\frac{2d}{2\alpha-1}}(\R^d)}\, \|h\|_{L^2(\R^d)}.
 \]
\end{lemma}
\begin{proof}
% Observe that by assumption $\alpha \in (\frac{1}{2},1)$ and 
% \[
%   \alpha  < \frac{d+2}{4},\quad \alpha < \frac{d+1}{2}.
% \]
By \Cref{la:Hdssformula},
\[
\begin{split}
 &\brac{H_{\Dso{}}(f h,g) - f  H_{\Dso}(h,g)}(x)\\
 =& c\int_{\R^d} \frac{\brac{f(x) h(x) - f(y) h(y)} (g(x)-g(y))}{|x-y|^{d+1}}\, dy-\int_{\R^d} \frac{f(x)\brac{h(x) - h(y)} (g(x)-g(y))}{|x-y|^{d+1}}\, dy\\
 =& c\int_{\R^d} \frac{\brac{\brac{f(x) h(x) - f(y) h(y)} -f(x)\brac{h(x) - h(y)}} (g(x)-g(y))}{|x-y|^{d+1}}\, dy\\
 =& c\int_{\R^d} \frac{\brac{f(x) - f(y)} (g(x)-g(y))}{|x-y|^{d+1}}\, h(y) dy\\
 \end{split}
\]
The claim now follows from \Cref{la:commies3b}, taking $p_1 = \frac{2d}{2\alpha-1}$, $p_2 = \frac{2d}{2\alpha-1}$ and $p_3 := \frac{2d}{d+2-4\alpha} > 2$, i.e. $\frac{dp_3}{d+(2\alpha -1)p_3} = 2$.
\end{proof}

\section{Decay estimate in time}\label{s:decay}
In this section we prove the main estimate for \Cref{th:main} which is
\begin{theorem}\label{th:detest}
Let $u,v: [0,T] \to \R^d$ be smooth solutions to the half-wave map equation \eqref{eq:half-wavemapseq}. Set 
\[
 \mathcal{E}(t) := \frac{1}{2} \brac{\|D_x (\vec{u}-\vec{v})(t)\|_{L^2(\R^d)}^2+\|\partial_t (\vec{u}-\vec{v})(t)\|_{L^2(\R^d)}^2}
\]
Then, for any $\alpha > 1$,
\[
 \dot{\mathcal{E}}(t) \leq  \Sigma(t) \mathcal{E}(t),
\]
where for any $\alpha >1$ we can estimate
\begin{equation}\label{eq:th:detest:est}
\begin{split}
 \Sigma(t) \aleq & \|\Ds{\alpha} \vec{u}(t)\|_{L^{\frac{2d}{2\alpha -1}}(\R^d)}^2 + \|\Ds{\alpha} \vec{v}(t)\|_{L^{\frac{2d}{2\alpha -1}}(\R^d)}^2\\
 &+\|\Dso{} \vec{u}(t)\|_{L^{(2d,2)}(\R^d)}^2 + \|\Dso{} \vec{v}(t)\|_{L^{(2d,2)}(\R^d)}^2\\
 \end{split}
\end{equation}
\end{theorem}

It remains to prove \Cref{th:detest}.
\subsection{Proof of Theorem~\ref{th:detest}}
We observe
\[
\begin{split}
\dot{\mathcal{E}}(t) =& \frac{1}{2} \frac{d}{dt}\brac{\|\nabla (\vec{u}-\vec{v})\|_{L^2(\R^d)}^2+\|\partial_t (\vec{u}-\vec{v})\|_{L^2(\R^d)}^2}\\
 =&\int_{\R^d} \langle \partial_{tt} \vec{w} - \lap \vec{w}, \partial_t \vec{w}\rangle\\
\end{split}
 \]

So what we need to do is multiply the equation for $\partial_{tt} \vec{w}- \lap \vec{w}$ with $\partial_t \vec{w}$. From the equation \eqref{eq:half-wavemapseq} for $u$ and $v$, respectively, we find that
\begin{equation}\label{eq:partialttwsplit}
\begin{split}
 \partial_{tt} \vec{w} - \lap \vec{w} =& \vec{u} |\nabla \vec{u}|^2-\vec{v} |\nabla \vec{v}|^2\\
 & +\vec{v} |\Dso{}\vec{v}|^2- \vec{u} |\Dso{}\vec{u}|^2\\
 &+ \Dso{} \vec{u}\, \brac{\scpr{\vec{u}}{\Dso{} \vec{u}}}-\Dso{} \vec{v}\, \brac{\vec{v} \cdot \Dso{} \vec{v}}\\
 &+ \vec{u} \wedge [\Dso{}, \vec{u} \wedge](\Dso{} \vec{u}) -\vec{v} \wedge  [\Dso{}, \vec{v} \wedge ](\Dso{} \vec{v}).
 \end{split}
\end{equation}
Here we recall the commutator notation
\[
 [T,f](g) = T(fg)-fT(g).
\]

We will prove the estimate of \Cref{th:detest} by estimating each line in \eqref{eq:partialttwsplit}, which will become increasingly more challenging, the last line being the most involved estimate. Having said that, the difficulties are mostly of algebraic nature, and the actual estimates rely on the fractional Leibniz rule discussed in \Cref{s:prelim}.

\medskip 

\underline{Repeating estimates}

\medskip 

Throughout the remainder of the section we will use \Cref{la:nablalaps} implicitly -- without further mentioning. 

Moreover, observe that for $\frac{1}{2}<\alpha_1 < \alpha_2< d+\frac{1}{2}$ we have from Sobolev embedding, \Cref{la:sobolev},
\[
  \|\Ds{\alpha_1} \vec{u}\|_{L^{\frac{2d}{2\alpha_1-1}} (\R^d)} \aleq                                                                                                                                                                                                             
  \|\Ds{\alpha_2} \vec{u}\|_{L^{\frac{2d}{2\alpha_2-1}} (\R^d)} 
  \]
In particular for any $\alpha \in [1,d+\frac{1}{2})$,
\[
  \|\Dso \vec{u}\|_{L^{2d} (\R^d)} \aleq                                                                                                                                                                                                             
  \|\Ds{\alpha} \vec{u}\|_{L^{\frac{2d}{2\alpha-1}} (\R^d)} 
  \]
This will be also used frequently and implicitly -- in particular to obtain the estimate in \Cref{th:detest} from the lemmata below.

\medskip

\underline{Estimating the first line of \eqref{eq:partialttwsplit}}

\medskip

We begin with the following estimate which is proven in Shatah-Struwe \cite{SS02}. 
\begin{lemma}\label{la:firstline}
For $d \geq 3$,
\[
\begin{split}
&\abs{\int_{\R^d}  \langle \vec{u}\, |\nabla \vec{u}|^2-\vec{v} \, |\nabla \vec{v}|^2,\, \partial_t (\vec{u}-\vec{v}) \rangle}\\
\aleq&\brac{\|\nabla (\vec{u}-\vec{v})\|_{L^2(\R^d)}^2+\|\partial_t (\vec{u}-\vec{v})\|_{L^2(\R^d)}^2}\, \brac{\|\Dso \vec{u}\|_{L^{2d}(\R^d)}^2+\|\Dso \vec{v}\|_{L^{2d}(\R^d)}^2}
\end{split}
 \]
\end{lemma}

\medskip

\underline{Estimating the Second line of \eqref{eq:partialttwsplit}}

\medskip 

In a similar spirit to \Cref{la:firstline} we can also obtain

\begin{lemma}\label{la:secondline}
For $d\geq 3$ we have
\[
\begin{split}
& \abs{\int_{\R^d}  \langle \vec{u} |\Dso{}\vec{u}|^2-\vec{v} |\Dso{}\vec{v}|^2,\, \partial_t (\vec{u}-\vec{v}) \rangle} \\
\aleq&\brac{\|\nabla (\vec{u}-\vec{v})\|_{L^2(\R^d)}^2+\|\partial_t (\vec{u}-\vec{v})\|_{L^2(\R^d)}^2}\, \brac{\|\Dso \vec{u}\|_{L^{2d}(\R^d)}^2+\|\Dso \vec{v}\|_{L^{2d}(\R^d)}^2}
\end{split}
 \]
\end{lemma}
\begin{proof}
We split
\begin{equation}\label{eq:split:1}
\begin{split}
 &\vec{u} |\Dso{}\vec{u}|^2-\vec{v} |\Dso{} \vec{v}|^2\\
 =& (\vec{u}-\vec{v})|\Dso{} \vec{u}|^2 + \vec{v}\langle \Dso{} (\vec{u}-\vec{v}),\Dso{} \vec{u}\rangle +\vec{v}\langle \Dso{} \vec{v},\Dso{} (\vec{u}-\vec{v})\rangle
 \end{split}
\end{equation}
From H\"older's inequality and Sobolev inequality, \Cref{la:sobolev},
\[
\begin{split}
 &\abs{\int_{\R^d}  |\Dso{} \vec{u}|^2 \langle \vec{u}-\vec{v}, \partial_t (\vec{u}-\vec{v}) \rangle}\\
 \aleq& \|\Dso{}\vec{u} \|_{L^{2d}(\R^d)}^2 \|\vec{u}-\vec{v}\|_{L^{\frac{2d}{d-2}}(\R^d)}\, \|\partial_t (\vec{u}-\vec{v})\|_{L^2(\R^d)}\\
  \aleq& \|\Dso{}\vec{u} \|_{L^{2d}(\R^d)}^2 \|\nabla (\vec{u}-\vec{v})\|_{L^{2}(\R^d)}\, \|\partial_t (\vec{u}-\vec{v})\|_{L^2(\R^d)}\\
  \aleq& \|\Dso{}\vec{u} \|_{L^{2d}(\R^d)}^2 \brac{\|\nabla (\vec{u}-\vec{v})\|_{L^{2}(\R^d)}^2 + \|\partial_t (\vec{u}-\vec{v})\|_{L^2(\R^d)}^2}.
 \end{split}
\]
This provides the desired estimate for the first term in \eqref{eq:split:1}.

The second and third term in \eqref{eq:split:1} are very similar, we only estimate the second one. Here we use the trick from \cite{SS02} that they used to obtain \Cref{la:firstline}: Since $\vec{u} \cdot \partial_t \vec{u}  = 0$ and $\vec{v} \cdot \partial_t \vec{v}=0$
\begin{equation}\label{eq:struwetrick}
\langle \vec{v}, \partial_t (\vec{u}-\vec{v}) \rangle = -\langle \vec{u}-\vec{v}, \partial_t \vec{u}  \rangle.
\end{equation}
Using that $u$ solves the half-wave map equation \eqref{eq:half-wavemapseq} and $|\vec{u}| \equiv 1$ we conclude
\[
 \abs{\langle \vec{v}, \partial_t (\vec{u}-\vec{v}) \rangle} \aleq \abs{\vec{u}-\vec{v}}\, |\Dso{}\vec{u}|.
\]
Thus using H\"older inequality and Sobolev inequality, \Cref{la:sobolev}, as before,
\[
\begin{split}
 &\abs{\int_{\R^d}  \langle \Dso{} (\vec{u}-\vec{v}),\Dso{} \vec{u}\rangle \langle \vec{v}, \partial_t (\vec{u}-\vec{v}) \rangle}\\
 \leq&\int_{\R^d}  |\Dso{} (\vec{u}-\vec{v})|\, |\Dso{} \vec{u}|\, \abs{\vec{u}-\vec{v}}\, |\Dso{}\vec{u}| \\
 \aleq& \|\Dso{} \vec{u}\|_{L^{2d}(\R^d)}^2\, \|\nabla (\vec{u}-\vec{v})\|_{L^2(\R^d)}\, \|\nabla (\vec{u}-\vec{v})\|_{L^2(\R^d)}\\
 \aeq& \|\nabla \vec{u}\|_{L^{2d}(\R^d)}^2\, \|\nabla (\vec{u}-\vec{v})\|_{L^2(\R^d)}^2.
 \end{split}
\]
\end{proof}

\medskip

\underline{Estimating the third line of \eqref{eq:partialttwsplit}}

\medskip

We recall our notation for the Leibniz rule operator
\[
 H_{T}(a,b) = T(ab) - aTb - (Ta)b.
\]
Observe that since $|\vec{u}|^2 \equiv 1$ we have 
\begin{equation}\label{eq:ucdotu}
 \scpr{\vec{u}}{\Dso{} \vec{u}} = -\frac{1}{2} H_{\Dso{}} (\vec{u} \cdot,\vec{u}) \equiv  -\frac{1}{2} \sum_{i=1}^3\, H_{\Dso{}} (u^i,u^i)
\end{equation}
So we consider 
\begin{equation}\label{eq:part3}
\begin{split}
&\Dso{} \vec{u}\, \scpr{\vec{u}}{\Dso{} \vec{u}}-\Dso{} \vec{v}\, \scpr{\vec{v}}{\Dso{} \vec{v}}\\
=&-\frac{1}{2}\brac{\Dso{} \vec{u}\, H_{\Dso{}} (\vec{u} \cdot,\vec{u}) - \Dso{} \vec{v}\, H_{\Dso{}} (\vec{v}\cdot,v)}\\
=&-\frac{1}{2}\Dso{} (\vec{u}-\vec{v})\, H_{\Dso{}} (\vec{u} \cdot,\vec{u})\\
&-\frac{1}{2}\Dso{} \vec{v}\, H_{\Dso{}} ((\vec{u}-\vec{v})\cdot,\vec{u})\\
&-\frac{1}{2}\Dso{} \vec{v}\, H_{\Dso{}} (\vec{v}\cdot,\vec{u}-\vec{v})
\end{split}
\end{equation}

Regarding the first term in \eqref{eq:part3} we observe that this is a more complicated structure to estimate, since two terms including $\vec{u}-\vec{v}$ appear to the full differential order, and we are not aware of a trick in the spirit of \cite{SS02} that would change that. Instead we use the commutator structure of $H_{\Dso{}}(\cdot,\cdot)$.
\begin{lemma}
For $d \geq 2$ we have
\[
\begin{split}
& \abs{\int_{\R^d}  \langle \Dso{} (\vec{u}-\vec{v}), \partial_t (\vec{u}-\vec{v})\rangle H_{\Dso{}} (\vec{u} \cdot,\vec{u})} \\
\aleq&\|\nabla (\vec{u}-\vec{v})\|_{L^2(\R^d)}\,\|\partial_t (\vec{u}-\vec{v})\|_{L^2(\R^d)}\, \|\Dso \vec{u}\|_{L^{2d,2}}^2.
\end{split}
 \]
\end{lemma}
\begin{proof}
By H\"older's inequality,
\[
  \abs{\int_{\R^d}  \langle \Dso{} (\vec{u}-\vec{v}), \partial_t (\vec{u}-\vec{v})\rangle H_{\Dso{}} (\vec{u} \cdot,\vec{u})} \aleq \|\nabla (\vec{u}-\vec{v})\|_{L^2(\R^d)}\,\|\partial_t (\vec{u}-\vec{v})\|_{L^2(\R^d)} \|H_{\Dso{}} (\vec{u} \cdot,\vec{u})\|_{L^\infty(\R^d)}.
\]
By \eqref{eq:comm:3} we have 
\[
 \|H_{\Dso{}} (\vec{u} \cdot,\vec{u})\|_{L^\infty(\R^d)} \aleq \|\Dso{} \vec{u}\|_{L^{2d,2}(\R^d)}^2.
\]
\end{proof}

Also the second and third term of \eqref{eq:part3} are relatively straight-forward to estimate using the commutator structure of $H_{\Dso}$.
\begin{lemma}
For $d \geq 2$, we have
\[
\begin{split}
& \abs{\int_{\R^d}  H_{\Dso{}} ((\vec{u}-\vec{v})\cdot,\vec{u})\, \langle \Dso{} \vec{v}, \partial_t (\vec{u}-\vec{v}) \rangle} \\
&+ \abs{\int_{\R^d}   H_{\Dso{}} (\vec{v}\cdot,\vec{u}-\vec{v})\, \langle \Dso{} \vec{v}, \partial_t (\vec{u}-\vec{v}) \rangle}\\
\aleq&\brac{\|\nabla (\vec{u}-\vec{v})\|_{L^2(\R^d)}^2+\|\partial_t (\vec{u}-\vec{v})\|_{L^2(\R^d)}^2}\, \brac{\|\nabla \vec{u}\|_{L^{2d}(\R^d)}+\|\nabla \vec{v}\|_{L^{2d}(\R^d)}}^2.
\end{split}
 \]
\end{lemma}
\begin{proof}
We only consider the first term, the second follows from the same argument. By H\"older's inequality, and \eqref{eq:comm:4},
\[
\begin{split}
& \abs{\int_{\R^d}  H_{\Dso{}} ((\vec{u}-\vec{v})\cdot,\vec{u})\, \langle \Dso{} \vec{v}, \partial_t (\vec{u}-\vec{v}) \rangle} \\
\aleq&\|\partial_t(\vec{u}-\vec{v})\|_{L^2(\R^d)} \|\Dso{} \vec{v}\|_{L^{2d}(\R^d)}\, \|H_{\Dso{}} ((\vec{u}-\vec{v})\cdot,\vec{u})\|_{L^{\frac{2d}{d-1}}(\R^d)}\\
% \aleq&\|\partial_t(\vec{u}-\vec{v})\|_{L^2(\R^d)} \|\Dso{} \vec{v}\|_{L^{2d}(\R^d)}\, \||\Ds{1/3}(\vec{u}-\vec{v})|\|_{L^{\frac{6d}{3d-4}}(\R^d)}\, \|\Ds{2/3} \vec{u}\|_{L^{6d}(\R^d)}\\
\overset{\eqref{eq:comm:4}}{\aleq}&\|\partial_t(\vec{u}-\vec{v})\|_{L^2(\R^d)} \|\Dso{} \vec{v}\|_{L^{2d}(\R^d)}\, \||\Dso{}(\vec{u}-\vec{v})|\|_{L^{2}(\R^d)}\, \|\Dso{} \vec{u}\|_{L^{2d}(\R^d)}
\end{split}
 \]
We can conclude since $\|\Dso{}(\vec{u}-\vec{v})\|_{L^{2}(\R^d)} \aeq \|\nabla(\vec{u}-\vec{v})\|_{L^{2}(\R^d)}$.
\end{proof}
${}$

\medskip

\underline{Estimating the last line of \eqref{eq:partialttwsplit}}

\medskip

We still need to understand the estimates for
\begin{equation}\label{eq:part4}
\begin{split}
&\vec{u} \wedge [\Dso{}, \vec{u} \wedge](\Dso{} \vec{u}) -\vec{v} \wedge  [\Dso{}, \vec{v} \wedge ](\Dso{} \vec{v})\\
=&(\vec{u}-\vec{v}) \wedge [\Dso{}, \vec{u} \wedge](\Dso{} \vec{u})\\
&+\vec{v} \wedge  \brac{[\Dso{}, (\vec{u}-\vec{v}) \wedge](\Dso{} \vec{u})}\\
&+\vec{v} \wedge  \brac{[\Dso{}, \vec{v} \wedge ](\Dso{} (\vec{u}-\vec{v}))}\\
\end{split}
\end{equation}

We observe that we can estimate the first term of \eqref{eq:part4} assuming a bound on $\Ds{\alpha} \vec{u}$ and $\Ds{\alpha} \vec{v}$ for an arbitrarily small $\alpha > 1$.
\begin{lemma}
For $d\geq 3$, and any $\alpha \in (1,d+\frac{1}{2})$ we have
\[
\abs{\int_{\R^d} \langle (\vec{u}-\vec{v}) \wedge [\Dso{}, \vec{u} \wedge](\Dso{} \vec{u}), \partial_t(\vec{u}-\vec{v})\rangle } \aleq \brac{\|\nabla(\vec{u}-\vec{v})\|_{L^2(\R^d)}^2+\|\partial_t(\vec{u}-\vec{v})\|_{L^2(\R^d)}^2}\, \|\Ds{\alpha} \vec{u}\|_{L^{\frac{2d}{2\alpha -1}}(\R^d)}^2.
\]
\end{lemma}
\begin{proof}
As before, by Cauchy-Schwarz and H\"older's inequality
\[
\begin{split}
&\abs{\int_{\R^d} \langle (\vec{u}-\vec{v}) \wedge [\Dso{}, \vec{u} \wedge](\Dso{} \vec{u}), \partial_t(\vec{u}-\vec{v})\rangle } \\
\aleq &\|\vec{u}-\vec{v}\|_{L^{\frac{2d}{d-2}}(\R^d)}\, \|\partial_t (\vec{u}-\vec{v})\|_{L^2(\R^d)}\,
\left \|[\Dso{}, \vec{u} \wedge](\Dso{} \vec{u}) \right \|_{L^{d}(\R^d)}\\
\aleq&\brac{\|\nabla(\vec{u}-\vec{v})\|_{L^2(\R^d)}^2+\|\partial_t(\vec{u}-\vec{v})\|_{L^2(\R^d)}^2}\, \left \|[\Dso{}, \vec{u} \wedge](\Dso{} \vec{u}) \right \|_{L^{d}(\R^d)}.
\end{split}
\]
In the last line we used Sobolev embedding, \Cref{la:sobolev}. It remains to estimate the commutator term. Observe that $\Dso{} \vec{u} \wedge \Dso{} \vec{u} = 0$ and thus we can write
\[
 [\Dso{}, \vec{u} \wedge](\Dso{} \vec{u}) = H_{\Dso{}}(\vec{u} \wedge,\Dso{} \vec{u}).
\]
We apply commutator theory, more precisely \eqref{eq:comm:5}, and find that for any $\sigma \in (1/2,1)$
\[
\begin{split}
\|[\Dso{}, \vec{u} \wedge](\Dso{} \vec{u})\|_{L^{d}(\R^d)} \aleq& \|\Ds{\sigma} \vec{u}\|_{L^{\frac{2d}{2\sigma-1}}(\R^d)}\, \|\Ds{2-\sigma} \vec{u}\|_{L^{\frac{2d}{2(2-\sigma)-1}}(\R^d)}\\
\aleq& \|\Ds{2-\sigma} \vec{u}\|_{L^{\frac{2d}{2(2-\sigma)-1}}(\R^d)}^2.
\end{split}
\]
The last line is Sobolev embedding, \Cref{la:sobolev}. Observe that if we set $\alpha := 2-\sigma$ then $\alpha > 1$ and $\alpha \aeq 1$ if $\sigma \aeq 1$, and we can conclude.
\end{proof}

Next we consider the term $\vec{v} \wedge  \brac{[\Dso{}, (\vec{u}-\vec{v}) \wedge](\Dso{} \vec{u})}$ in \eqref{eq:part4}.

We observe 
\begin{equation}\label{eq:part4:split1}
\begin{split}
 \vec{v} \wedge  \brac{[\Dso{}, (\vec{u}-\vec{v}) \wedge](\Dso{} \vec{u})}
 =&\vec{v} \wedge  \brac{H_{\Dso}((\vec{u}-\vec{v}) \wedge,\Dso \vec{u})}\\
 &+\vec{v} \wedge  \brac{\Dso (\vec{u}-\vec{v}) \wedge\Dso \vec{u}}\\
 \end{split}
\end{equation}
We first establish the following estimate which estimates the first term on the right-hand side in \eqref{eq:part4:split1}
\begin{lemma}\label{la:weirdHest1}
 For any $\alpha \in (1,d+\frac{1}{2})$, and any $d \geq 3$,
\[
\|\vec{v} \wedge  \brac{H_{\Dso{}}\brac{(\vec{u}-\vec{v}) \wedge,\Dso{} \vec{u}}} \|_{L^{2}(\R^d)} \aleq \|\nabla (\vec{u}-\vec{v})\|_{L^{2}(\R^d)}\, \brac{\|\Ds{\alpha} \vec{u}\|_{L^{\frac{2d}{2\alpha-1}}}^2+\|\Ds{\alpha} \vec{v}\|_{L^{\frac{2d}{2\alpha-1}}}^2}
\]
in particular
\[
\begin{split}
&\int_{\R^d} \scpr{\vec{v} \wedge  \brac{H_{\Dso{}}\brac{(\vec{u}-\vec{v}) \wedge,\Dso{} \vec{u}}}} {\partial_{t}(\vec{u}-\vec{v})} \\
\aleq &
\|\Dso{} (\vec{u}-\vec{v})\|_{L^{2}(\R^d)}\, \brac{\|\Ds{\alpha} \vec{u}\|_{L^{\frac{2d}{2\alpha-1}}}^2+\|\Ds{\alpha} \vec{v}\|_{L^{\frac{2d}{2\alpha-1}}}^2}\, \|\partial_t (\vec{u}-\vec{v})\|_{L^2(\R^d)}
\end{split}
\]
\end{lemma}
\begin{proof} We recall the formula
\[
\vec{a} \wedge (\vec{b} \wedge \vec{c})=\vec{b} (\vec{c} \cdot \vec{a}) - \vec{c} (\vec{a} \cdot \vec{b}) 
\]
so 
\begin{equation}\label{eq:laweirdHest1:2435}
\begin{split}
 &\brac{\vec{v} \wedge  \brac{H_{\Dso{}}\brac{(\vec{u}-\vec{v}) \wedge,\Dso{} \vec{u}}}}^{i}\\
 =& \sum_{j=1}^3\brac{v^j H_{\Dso{}}\brac{(u-v)^i ,\Dso{} u^j}}-\brac{v^j H_{\Dso{}}\brac{(u-v)^j ,\Dso{} u^i}}\\
 = & \sum_{j=1}^3(v-u)^{{j}} H_{\Dso{}}\brac{(u-v)^i ,\Dso{} u^{j}}\\
 &-\frac{1}{2}{(u-v)^j H_{\Dso{}}\brac{(u-v)^j ,\Dso{} u^i}}\\
  &+\frac{1}{2}\sum_{j=1}^3\brac{(u+v)^j H_{\Dso{}}\brac{(u-v)^j ,\Dso{} u^i}}\\
  &+\sum_{j=1}^3\brac{u^{{j}} H_{\Dso{}}\brac{(u-v)^i ,\Dso{} u^{j}}}.
 \end{split}
\end{equation}
Take any $\alpha \in (1,d+\frac{1}{2})$ and set $\sigma := 2-\alpha$. Then $\frac{2d^2}{d-(1-\sigma)2d} \in (1,\infty)$, and by \eqref{eq:comm:5} we can estimate the first two terms of \eqref{eq:laweirdHest1:2435},
\begin{equation}\label{eq:laweirdHest1:656}
\begin{split}
 &\|\abs{\vec{v}-\vec{u}} \abs{H_{\Dso{}}\brac{(u-v)^i ,\Dso{} u^j}}\|_{L^{2}(\R^d)}\\
 \aleq& \|\vec{u}-\vec{v}\|_{L^{\frac{2d}{d-2}}(\R^d)}\, \|\Ds{\sigma} (u-v)^i\|_{L^{\frac{2d^2}{d-(1-\sigma)2d}(\R^d)}} \|\Ds{1+\sigma} u^j\|_{L^{\frac{2d}{2(2-\sigma)-1}}}\\
 \aleq& \|\Dso{} (\vec{u}-\vec{v})\|_{L^{2}(\R^d)}\, \brac{\|\Ds{\alpha} \vec{u}\|_{L^{\frac{2d}{2\alpha-1}}}^2+\|\Ds{\alpha} \vec{v}\|_{L^{\frac{2d}{2\alpha-1}}}^2}.
 \end{split}
\end{equation}
For the third term in \eqref{eq:laweirdHest1:2435} observe that for fixed $i$, using \Cref{la:Hdssformula},
\[
\begin{split}
 &\sum_j \frac{1}{2}\brac{(u+v)^j H_{\Dso{}}\brac{(u-v)^j ,\Dso{} u^i}}(x)\\
 =& c \int_{\R^d} \frac{\brac{\Dso{} u^i(x)-\Dso{} u^i(y)}\, \langle (\vec{u+v})(x),\brac{(\vec{u}-\vec{v})(x)-(\vec{u}-\vec{v})(y)}\rangle } {|x-y|^{d+1}}dy.
 \end{split}
\]
Setting $\vec{a} := \vec{u+v}$ and $\vec{b} := \vec{u}-\vec{v}$ we recall that $\vec{a}(x) \cdot \vec{b}(x) = |\vec{u}(x)|^2-|\vec{v}(x)|^2 = 0$, and thus we can perform a discrete version of the trick in \eqref{eq:struwetrick},
\[
\begin{split}
 \vec{a}(x) \cdot (\vec{b}(x)-\vec{b}(y)) =& (\vec{a}(x)-\vec{a}(y))\cdot (\vec{b}(x)-\vec{b}(y)) +\vec{a}(y)\cdot  (\vec{b}(x)-\vec{b}(y))\\
 =&(\vec{a}(x)-\vec{a}(y))\cdot (\vec{b}(x)-\vec{b}(y)) +\vec{a}(y)\cdot  \vec{b}(x)\\
 =&(\vec{a}(x)-\vec{a}(y))\cdot (\vec{b}(x)-\vec{b}(y)) -\brac{\vec{a}(x)-\vec{a}(y)}\cdot  \vec{b}(x).
 \end{split}
\]
This implies that 
\[
\begin{split}
&\sum_j \frac{1}{2}\brac{(u+v)^j H_{\Dso{}}\brac{(u-v)^j ,\Dso{} u^i}}(x) \\
=&-\sum_j \frac{1}{2}\brac{(u{-}v)^j H_{\Dso{}}\brac{(u{+}v)^j ,\Dso{} u^i}}(x) \\
&+ c \int_{\R^d} \frac{\brac{\Dso{} u^i(x)-\Dso{} u^i(y)}\, \scpr{(\vec{u+v})(x)-(\vec{u+v})(y)}{{(\vec{u}-\vec{v})(y)-(\vec{u}-\vec{v})(y)}} } {|x-y|^{d+1}}dy.
\end{split}
\]
For the first term, just as above for \eqref{eq:laweirdHest1:656},
\[
 \left \| \frac{1}{2}(u-v)^j(x) H_{\Dso{}}\brac{(u{+}v)^j ,\Dso{} u^i}(x) \right \|_{L^{2}(\R^d)} \aleq \|\Dso{} (\vec{u}-\vec{v})\|_{L^{2}(\R^d)}\, \brac{\|\Ds{\alpha} \vec{u}\|_{L^{\frac{2d}{2\alpha-1}}}^2+\|\Ds{\alpha} \vec{v}\|_{L^{\frac{2d}{2\alpha-1}}}^2}
\]
For the second term we use \Cref{la:commies3}. Take $\sigma, \theta > 0$ such that $\sigma + \theta \in (0,1)$. Then, (if $d \geq 2$ we can take $\sigma, \theta$ small enough to make any of the norms below finite),
\[
\begin{split}
 &\left \|\int_{\R^d} \frac{\brac{\Dso{} u^i(x)-\Dso{} u^i(y)}\, \scpr{(\vec{u}+\vec{v})(x)-(\vec{u}+\vec{v})(y)}{(\vec{u}-\vec{v})(y)-(\vec{u}-\vec{v})(y)}} {|x-y|^{d+1}}dy\right \|_{L^2(\R^d,dx)}\\
 \aleq &\|\Ds{1+\sigma} u^i\|_{L^{\frac{2d}{2(1+\sigma)-1}}}\, \|\Ds{1-\theta-\sigma} (\vec{u}+\vec{v})\|_{\frac{2d^2}{d-(\theta+\sigma)2d}}\, \|\Ds{\theta} (\vec{u}-\vec{v})\|_{L^{\frac{d2}{d-(1-\theta) 2}}}\\
 \aleq&\|\Ds{\alpha} u^i\|_{L^{\frac{2d}{2(\alpha)-1}}}\, \brac{\|\Dso{} \vec{u}\|_{L^{2d}(\R^d)}+\|\Dso{} \vec{v}\|_{L^{2d}(\R^d)}}\, \|\Dso{} (\vec{u}-\vec{v})\|_{L^{2}(\R^d)}.
 \end{split}
\]
This establishes the right estimate for the second term in \eqref{eq:laweirdHest1:2435}.

For the last term in \eqref{eq:laweirdHest1:2435} it remains we consider, again using \Cref{la:Hdssformula},
\[
\begin{split}
&\sum_{j=1}^3\brac{u^{j} H_{\Dso{}}\brac{(u-v)^i ,\Dso{} u^j}}(x)\\
=&\int_{\R^n} \frac{\brac{(u-v)^i(x)-(u-v)^i(y)}\, \left \langle \Dso{} \vec{u}(x)-\Dso{} \vec{u}(y), \vec{u}(x) \right \rangle }{|x-y|^{d+1}}\, dy.
\end{split}
\]
Now we write 
\[
\begin{split}
&\scpr{\Dso{} \vec{u}(x)-\Dso{} \vec{u}(y)}{\vec{u}(x)}\\
=&\scpr{\vec{u}(x) }{\Dso{} \vec{u}(x)}-\scpr{\vec{u}(y) }{ \Dso{} \vec{u}(y)}+\scpr{\vec{u}(y) -\vec{u}(x)}{\Dso{} \vec{u}(y)}\\
=&\scpr{\vec{u}(x) }{ \Dso{} \vec{u}(x)}-\scpr{\vec{u}(y) ,\Dso{} \vec{u}(y)}+\scpr{\vec{u}(y) -\vec{u}(x)}{\Dso{}\vec{u}(y)-\Dso{}\vec{u}(x)}+\scpr{\vec{u}(y)-\vec{u}(x)}{\Dso{}\vec{u}(x)}\\
=&\scpr{\vec{u}(x) }{\Dso{} \vec{u}(x)}-\scpr{\vec{u}(y) }{\Dso{} \vec{u}(y)}+\scpr{\vec{u}(x) -\vec{u}(y)}{\Dso{}\vec{u}(x)-\Dso{}\vec{u}(y)}-\scpr{\vec{u}(x)-\vec{u}(y)}{\Dso{}\vec{u}(x)}\\
\end{split}
\]
Thus we have for $i=1,2,3$,
\begin{equation}\label{eq:L2est:325363}
\begin{split}
&\sum_{j=1}^3{u^{j} H_{\Dso{}}\brac{(u-v)^i ,\Dso{} u^j}}(x)\\
=&H_{\Dso{}}\brac{(u-v)^i , \scpr{\vec{u} }{\Dso{} \vec{u}}}(x)\\
&-\sum_{j=1}^3 \Dso{} u^j(x)\, H_{\Dso{}}\brac{(u-v)^i , u^j}(x)\\
&+\int_{\R^n} \frac{\brac{(u-v)^i(x)-(u-v)^i(y)}\, \scpr{\Dso{} \vec{u}(x)-\Dso{} \vec{u}(y)}{\vec{u}(x)-\vec{u}(y)}} {|x-y|^{d+1}}\, dy.
\end{split}
\end{equation}
We estimate the first term in \eqref{eq:L2est:325363}. Applying first \eqref{eq:comm:5}, for any small $\sigma > 0$, Sobolev embedding, \Cref{la:sobolev}, and then \eqref{eq:comm:1},
\[
\begin{split}
 \|H_{\Dso{}}\brac{(u-v)^i , \scpr{\vec{u}}{\Dso{} \vec{u}}}\|_{L^2(\R^d)} \aleq& \|\Ds{1-\sigma} (\vec{u}-\vec{v})\|_{L^{\frac{d2}{d-\sigma 2}}(\R^d)}\, \|\Ds{\sigma} \underbrace{(\scpr{\vec{u}}{\Dso{} \vec{u}})}_{=H_{\Dso{}}(\vec{u}\cdot,\vec{u})}\|_{L^{\frac{d}{\sigma}}(\R^d)}\\
 \aleq&\|\nabla (\vec{u}-\vec{v})\|_{L^2(\R^d)}\, \|\Dso{} \underbrace{(\scpr{\vec{u}}{\Dso{} \vec{u}})}_{=H_{\Dso{}}(\vec{u}\cdot,\vec{u})}\|_{L^{d}(\R^d)}\\
%  \aleq&\|\nabla (\vec{u}-\vec{v})\|_{L^2(\R^d)}\, \|\Ds{\frac{\sigma+1}{2}} \vec{u}\|_{L^{\frac{2d}{\sigma}}(\R^d)}^2\\
 \overset{\eqref{eq:comm:1}}{\aleq}&\|\nabla (\vec{u}-\vec{v})\|_{L^2(\R^d)}\, \|\Dso{}\vec{u}\|_{L^{2d}(\R^d)}^2.
 \end{split}
\]
For the second term in \eqref{eq:L2est:325363}, by  \eqref{eq:comm:5},
\[
\begin{split}
 &\|\abs{\Dso{} \vec{u}} \abs{H_{\Dso{}}\brac{(u-v)^i , \vec{u}}}\|_{L^2(\R^d)}\\
 \aleq& \|\Dso{} \vec{u}\|_{L^{2d}(\R^d)} \|H_{\Dso{}}\brac{(u-v)^i , \vec{u}}\|_{L^{\frac{2d}{d-1}}(\R^d)}\\
 \aleq& \|\Dso{} \vec{u}\|_{L^{2d}(\R^d)} \|\Ds{1-\sigma}(\vec{u}-\vec{v})\|_{L^{\frac{2d}{d-\sigma 2}}(\R^d)} \|\Ds{\sigma} \vec{u}\|_{L^{\frac{2d^2}{d-(1-\sigma)2d}}(\R^d)}\\
 \aleq& \|\Dso{} \vec{u}\|_{L^{2d}(\R^d)} \|\Dso{}(\vec{u}-\vec{v})\|_{L^{2}(\R^d)} \|\Dso{} \vec{u}\|_{L^{2d}(\R^d)}.
\end{split}
\]
For the last term in \eqref{eq:L2est:325363}, using \Cref{la:commies3}, for $\sigma, \theta > 0$ such that $\sigma + \theta < 1$,
 \[
 \begin{split}
  &\left \|\int_{\R^n} \frac{\brac{(u-v)^i(x)-(u-v)^i(y)}\, \brac{\Dso{} \vec{u}(x)-\Dso{} \vec{u}(y)} \cdot \brac{\vec{u}(x)-\vec{u}(y)}}{|x-y|^{d+1}}\, dy\right \|_{L^2(\R^d,dx)}\\
  \aleq& \|\Ds{\sigma} (\vec{u}-\vec{v})\|_{L^{\frac{2d}{d-(1-\sigma)2}}(\R^d)}\, \|\Ds{\theta} \vec{u}\|_{L^{\frac{2d^2}{d-(1-\theta)2d}}(\R^d)}\, \|\Ds{2-\sigma-\theta} \vec{u}\|_{L^{\frac{2d}{2(2-\sigma-\theta)-1}}(\R^d)}\\
  \aleq&\|\nabla (\vec{u}-\vec{v})\|_{L^2(\R^d)}\, \|\Ds{\alpha} \vec{u}\|_{L^{\frac{2d}{2\alpha-1}}}^2,
  \end{split}
 \]
where we have set $\alpha := 2-\sigma - \theta$.

This conclude the estimate of the last term of \eqref{eq:laweirdHest1:2435}.
\end{proof}

In order to estimate the second term in \eqref{eq:part4:split1} we observe that since $\vec{u}$ and $\vec{v}$ are both solutions to the halfwave equation \eqref{eq:half-wavemapseq} we have
\[
\begin{split}
 \partial_t (\vec{u}-\vec{v}) =& \vec{u} \wedge \Dso \vec{u} - \vec{v} \wedge  \Dso \vec{v}\\
 =& \brac{\vec{u}-\vec{v}} \wedge \Dso \vec{u} + \vec{v} \wedge  \Dso (\vec{u}-\vec{v}).
\end{split}
 \]

Consequently, we split the estimate for the second term in \eqref{eq:part4:split1}
\begin{equation}\label{eq:part4:split1:splitagain}
\begin{split}
 &\int_{\R^d} \brac{\vec{v} \wedge  \brac{\Dso (\vec{u}-\vec{v}) \wedge\Dso \vec{u}}} \cdot \partial_t (\vec{u}-\vec{v})\\
 =&\int_{\R^d} \vec{v} \wedge  \brac{\Dso (\vec{u}-\vec{v}) \wedge\Dso \vec{u}} \cdot \brac{(\vec{u}-\vec{v}) \wedge \Dso{} \vec{u}}\\
 &+\int_{\R^d} \brac{\vec{v} \wedge  \brac{\Dso (\vec{u}-\vec{v}) \wedge\Dso \vec{u}}}\cdot \brac{\vec{v} \wedge  \Dso{}(\vec{u}-\vec{v})}\\
\end{split}
 \end{equation}

The first term in \eqref{eq:part4:split1:splitagain} can be estimated with H\"older and Sobolev inequality, \Cref{la:sobolev},
\begin{lemma}
For $d \geq 3$ we have 
\[
 \abs{\int_{\R^d} \scpr{\vec{v} \wedge  \brac{\Dso (\vec{u}-\vec{v}) \wedge\Dso \vec{u}}}{(\vec{u}-\vec{v}) \wedge \Dso{} \vec{u}} } \aleq \|\Dso{} (\vec{u}-\vec{v})\|_{L^{2}(\R^d)}^2\, \|\Dso \vec{u}\|_{L^{2d}(\R^d)}^2
\]
\end{lemma}
\begin{proof}
We have
\[
\begin{split}
 &\abs{\int_{\R^d} \brac{\vec{v} \wedge  \brac{\Dso (\vec{u}-\vec{v}) \wedge\Dso \vec{u}} \cdot \brac{(\vec{u}-\vec{v}) \wedge \Dso{} \vec{u}}}}  \\
 \aleq&\|\vec{v}\|_{L^\infty(\R^d)}\, \|\Dso(\vec{u}-\vec{v})\|_{L^2(\R^d)}\, \|\Dso{} \vec{u} \|_{L^{2d}(\R^d)}  \|\vec{u}-\vec{v}\|_{L^{\frac{2d}{d-2}}(\R^d)} \|\Dso{}\vec{u}\|_{L^{2d}(\R^d)}. \\
  \end{split}
\]
We can conclude by Sobolev inequality, \Cref{la:sobolev}, observing also that $|\vec{v}| = 1$.
\end{proof}

For the last term from \eqref{eq:part4:split1:splitagain} we establish what can be interpreted as a fractional version and extension of the trick \eqref{eq:struwetrick} from \cite{SS02}.
\begin{lemma}\label{la:struwetrickfrac1}
Let $d\geq 3$, $|\vec{v}|=|\vec{u}| \equiv 1$ then for $\sigma \in [0,1)$,
\[
\|\Ds{\sigma} \brac{\vec{v} \cdot \Dso{} (\vec{u}-\vec{v})} \|_{L^{\frac{2d}{d-1+2\sigma}}(\R^d)} \aleq \|\nabla (\vec{u}-\vec{v})\|_{L^2(\R^d)}\, \brac{\|\Ds{1+\sigma} \vec{u}\|_{L^{\frac{2d}{2(1+\sigma)-1}}(\R^d)} + \|\Ds{1+\sigma} \vec{v}\|_{L^{\frac{2d}{2(1+\sigma)-1}}(\R^d)}}
\]
\end{lemma}
\begin{proof}
Since  $|\vec{u}| = |\vec{v}| = 1$,
\[
 (\vec{u}+\vec{v})\cdot (\vec{u}-\vec{v}) = |\vec{u}|^2-|\vec{v}|^2 = 0.
\]
Consequently,
\[
 (\vec{u}+\vec{v}) \cdot \Dso{} (\vec{u}-\vec{v}) = -(\Dso{} (\vec{u}+\vec{v}) )\cdot (\vec{u}-\vec{v}) -H_{\Dso{}} \brac{(\vec{u}+\vec{v})\cdot,\vec{u}-\vec{v}},
\]
That is
\[
\begin{split}
 \vec{v} \cdot \Dso{} (\vec{u}-\vec{v}) =& -\frac{1}{2} (\vec{u}-\vec{v}) \cdot \Dso{}(\vec{u}-\vec{v}) + \frac{1}{2} (\vec{u}+\vec{v}) \cdot \Dso{}(\vec{u}-\vec{v}) \\
 =&-\frac{1}{2} (\vec{u}-\vec{v}) \cdot \Dso{}(\vec{u}-\vec{v}) - \frac{1}{2} \brac{\Dso{} (\vec{u}+\vec{v}) \cdot (\vec{u}-\vec{v}) +H_{\Dso{}} ((\vec{u}+\vec{v})\cdot,\vec{u}-\vec{v})} \\
\end{split}
 \]
 
In order to estimate $\Ds{\sigma} \brac{\vec{v} \cdot \Dso{} (\vec{u}-\vec{v})}$ we use \Cref{la:wugsd8gfus0df}.

We can conclude.
\end{proof}

The last term from \eqref{eq:part4:split1:splitagain} is estimated in the following
\begin{lemma}
For $d \geq 3$
\[
\abs{\int_{\R^d} \brac{\vec{v} \wedge  \brac{\Dso{}(\vec{u}-\vec{v}) \wedge \Dso{} \vec{u}}} \cdot \brac{ \vec{v} \wedge  \Dso{}(\vec{u}-\vec{v})}} \aleq \|\nabla (\vec{u}-\vec{v})\|_{L^2(\R^d)}^2 \ \brac{\|\Dso{} \vec{v}\|_{L^{2d}(\R^d)}^2 +\|\Dso{} \vec{u}\|_{L^{2d}(\R^d)}^2} \\
\]
\end{lemma}
\begin{proof}
Recall the formula
\[
 (\vec{a} \wedge \vec{b}) \cdot (\vec{c} \wedge \vec{d}) = (\vec{a} \cdot \vec{c}) (\vec{b} \cdot \vec{d}) - (\vec{a} \cdot \vec{d}) (\vec{c} \cdot \vec{b})
\]
so in particular
\[
 (\vec{a} \wedge \vec{b}) \cdot (\vec{a} \wedge \vec{d}) = |\vec{a}|^2 (\vec{b} \cdot \vec{d}) - (\vec{a} \cdot \vec{d}) (\vec{a} \cdot \vec{b}).
\]
For $\vec{a} :=\vec{v}$, $\vec{b} = {\Dso{}(\vec{u}-\vec{v}) \wedge \Dso{} \vec{u}}$, $\vec{d} = \Dso{}(\vec{u}-\vec{v})$ we observe that $\vec{b} \cdot \vec{d} = 0$, so
\[
 \brac{\vec{v} \wedge  \brac{\Dso{}(\vec{u}-\vec{v}) \wedge \Dso{} \vec{u}}} \cdot \brac{ \vec{v} \wedge  \Dso{}(\vec{u}-\vec{v})}  =
 \scpr{\vec{v} }{\brac{\Dso{}(\vec{u}-\vec{v}) \wedge \Dso{} \vec{u}}}\ \scpr{\vec{v}}{\Dso{}(\vec{u}-\vec{v})}
\]
We combine this observation with the estimate of \Cref{la:struwetrickfrac1} for $\sigma = 0$ and conclude 
% 
% 
% We conclude that (for any $\sigma \in (\frac{1}{2},1)$, so that $\frac{2d^2}{d-(1-\sigma)2d} \in (1,\infty)$), using \Cref{la:commie},
\[
 \begin{split}
& \int_{\R^d} \abs{\brac{\vec{v} \wedge  \brac{\Dso{}(\vec{u}-\vec{v}) \wedge \Dso{} \vec{u}}} \cdot \brac{ \vec{v} \wedge  \Dso{}(\vec{u}-\vec{v})}}\\
 \aleq & \|\vec{v}\|_{L^\infty(\R^d)}\, \| \Dso (\vec{u}-\vec{v})\|_{L^2(\R^d)}\, \| \Dso \vec{u} \|_{L^{2d}(\R^d)}\, \|\scpr{\vec{v}}{\Dso{}(\vec{u}-\vec{v})}\|_{L^{\frac{2d}{d-1}}(\R^d)}\\
 \aleq&\| \Dso (\vec{u}-\vec{v})\|_{L^2(\R^d)}\, \| \Dso \vec{u} \|_{L^{2d}(\R^d)}\, \|\nabla (\vec{u}-\vec{v})\|_{L^2(\R^d)}\, \brac{\|\Dso \vec{u}\|_{L^{2d}(\R^d)} + \|\Dso \vec{v}\|_{L^{2d}(\R^d)}}.
 \end{split}
\]

We can conclude.
\end{proof} 
By now we have estimated the second term in \eqref{eq:part4:split1}, which in turn concludes the desired estimate for the second term on the right-hand side of \eqref{eq:part4}.

The last term we need to understand is the last term of \eqref{eq:part4}, namely we are interested in an estimate for
\[
 \int_{\R^d} \vec{v} \wedge  \brac{[\Dso{}, \vec{v} \wedge ](\Dso{} (\vec{u}-\vec{v}))} \cdot \partial_t (\vec{u}-\vec{v})
\]
Using again the formula
\[
\vec{a} \wedge (\vec{b} \wedge \vec{c})=\vec{b} (\vec{c} \cdot \vec{a}) - \vec{c} (\vec{a} \cdot \vec{b}) 
\]
we have 
\begin{equation}\label{eq:verylastpart:split1}
\begin{split}
 & \brac{\vec{v} \wedge  \brac{[\Dso{}, \vec{v} \wedge ](\Dso{} (\vec{u}-\vec{v}))}}^i \\
 =& \sum_{j=1}^3\brac{v^j \brac{[\Dso{}, v^i ](\Dso{} (u-v)^j)}-\brac{[\Dso{}, v^j ](\Dso{} (u-v)^i)}}\\
= & \sum_{j=1}^3v^j \Dso{}v^i\, \Dso{} (u-v)^j\\
&-\sum_{j=1}^3v^j \Dso{}v^j\, \Dso{} (u-v)^i\\
  &+\sum_{j=1}^3v^j H_{\Dso}(v^i ,\Dso{} (u-v)^j)\\
  &-\sum_{j=1}^3 v^jH_{\Dso}(v^j ,\Dso{} (u-v)^i).
\end{split}
\end{equation}

We estimate the first term in  \eqref{eq:verylastpart:split1}.
\begin{lemma}
For $d \geq 3$,
\[
\begin{split}
 &\left \|\sum_{j=1}^3v^j \Dso{}v^i\, \Dso{} (u-v)^j\right \|_{L^2(\R^d)}\\
 \aleq&\brac{\|\Dso{} \vec{v}\|_{L^{2d}(\R^d)}^2+\|\Dso{} \vec{u}\|^2_{L^{2d}(\R^d)}}\, \|\nabla (\vec{u}-\vec{v})\|_{L^2(\R^d)}.
 \end{split}
\]
and in particular
\[
\begin{split}
 &\abs{\int_{\R^d} \sum_{i,j=1}^3v^j \Dso{}v^i\, \Dso{} (u-v)^j\, \partial_t (u-v)^i }\\
 \aleq & \brac{\|\Dso{} \vec{v}\|_{L^{2d}(\R^d)}^2+\|\Dso{} \vec{u}\|^2_{L^{2d}(\R^d)}}\, \brac{\|\nabla (\vec{u}-\vec{v})\|_{L^2(\R^d)}^2+\|\partial_t (\vec{u}-\vec{v})\|_{L^2(\R^d)}^2}.
\end{split}
 \]

\end{lemma}
\begin{proof}
By \Cref{la:struwetrickfrac1} for $\sigma = 0$
\[
\begin{split}
 &\left \|\sum_{j=1}^3v^j \Dso{}v^i\, \Dso{} (u-v)^j)\right \|_{L^2(\R^d)}\\
 \aleq &\|\Dso\vec{v}\|_{L^{2d(\R^d)}}\, \|\scpr{\vec{v}}{\Dso{} (\vec{u}-\vec{v})}\|_{L^{\frac{2d}{d-1}}(\R^d)}\\
 \aleq &\|\Dso\vec{v}\|_{L^{2d(\R^d)}}\, \|\Dso (\vec{u}-\vec{v})\|_{L^{2}(\R^d)}\, \brac{\|\Dso \vec{u}\|_{L^{2d}(\R^d)}+\|\Dso \vec{v}\|_{L^{2d}(\R^d)}}.
\end{split}
\]
We can conclude.

\end{proof}

We estimate the second term on the right-hand side of \eqref{eq:verylastpart:split1}
\begin{lemma}
For $d \geq 2$,
\[
 \|\sum_{j=1}^3v^j \Dso{}v^j\ \Dso{} (u-v)^i\|_{L^2(\R^d)} \aleq \|\Dso{} \vec{v}\|_{L^{2d,2}(\R^d)}^2\, \|\nabla (\vec{u}-\vec{v})\|_{L^2(\R^d)}.
\]
In particular
\[
\begin{split}
&\sum_{i,j=1}^3\int_{\R^d}  {v^j \Dso{}v^j\, \Dso{} (u-v)^i}\, \partial_t (u-v)^i\\
\aleq& \|\Dso{} \vec{v}\|_{L^{2d,2}(\R^d)}^2\, \|\nabla (\vec{u}-\vec{v})\|_{L^2(\R^d)}\, \|\partial_t (\vec{u}-\vec{v})\|_{L^2(\R^d)}.
\end{split}
\]
\end{lemma}
\begin{proof}
By \eqref{eq:comm:3} we have
\[
 \|\vec{v} \cdot \Dso{}\vec{v}\|_{L^\infty} \aleq \|\Dso{} \vec{v}\|_{L^{2d,2}(\R^d)}^2
\]
which readily implies the claim by H\"older's inequality.
\end{proof}

We estimate the third term on the right-hand side of \eqref{eq:verylastpart:split1}
\begin{lemma}\label{la:firstguyisfine}
For any $\alpha \in (1,d+\frac{1}{2})$, $d \geq 2$
\[
 \|\sum_{j=1}^3v^j H_{\Dso}(v^i ,\Dso{} (u-v)^j)\|_{L^2(\R^n)} \aleq \|\nabla (\vec{u}-\vec{v})\|_{L^2(\R^d)}\, \|\Ds{\alpha} \vec{v}\|_{L^{\frac{2d}{2\alpha-1}}(\R^d)}^2.
\]
In particular we have 
\[
\begin{split}
 &\abs{\int \sum_{i,j=1}^3{v^j H_{\Dso}(v^i ,\Dso{} (u-v)^j)}\, \partial_t (u-v)^i}\\
 \aleq& \brac{\|\nabla (\vec{u}-\vec{v})\|_{L^2(\R^d)}^2+\|\partial_t (\vec{u}-\vec{v})\|_{L^2(\R^d)}^2}\, \|\Ds{\alpha} \vec{v}\|_{L^{\frac{2d}{2\alpha-1}}(\R^d)}^2.
 \end{split}
\]
 \end{lemma}
 \begin{proof}
We use \Cref{la:Hdssformula}, and have 
\[
H_{\Dso}(v^i ,\Dso{} (u-v)^j)(x) = \int_{\R^d} \frac{\brac{v^i(x)-v^i(y)}\, \brac{\Dso (u-v)^j(x) - \Dso (u-v)^j (y)}}{|x-y|^{d+1}} dy
\]
Thus,
\[
\begin{split}
&\sum_{j=1}^3 v^j(x) H_{\Dso}(v^i ,\Dso{} (u-v)^j)(x)\\
=& c\int_{\R^d} \frac{\brac{v^i(x)-v^i(y)}\, \brac{\scpr{\vec{v}(x)}{\Dso (\vec{u}-\vec{v})(x)} - \scpr{\vec{v}(x)}{\Dso (\vec{u}-\vec{v}) (y) } }}{|x-y|^{d+1}} dy\\
=& c\int_{\R^d} \frac{\brac{v^i(x)-v^i(y)}\, \brac{\scpr{\vec{v}(x)}{\Dso (\vec{u}-\vec{v})(x)} - \scpr{\vec{v}(y)}{\Dso (\vec{u}-\vec{v}) (y) } }}{|x-y|^{d+1}} dy\\
& -c\int_{\R^d} \frac{\brac{v^i(x)-v^i(y)}\, \scpr{\vec{v}(x)-\vec{v}(y)}{\Dso (\vec{u}-\vec{v}) (y)  }}{|x-y|^{d+1}} dy\\
=&H_{\Dso}\brac{v^i, \langle \vec{v},\Dso (\vec{u}-\vec{v})\rangle}(x)\\
&-c\int_{\R^d} \frac{\brac{v^i(x)-v^i(y)}\, \scpr{\vec{v}(x)-\vec{v}(y)}{\Dso (\vec{u}-\vec{v}) (y)  }}{|x-y|^{d+1}} dy.
\end{split}
\]
The first term we can estimate with the help of \Cref{la:struwetrickfrac1}, for any small $\sigma \in (0,\frac{1}{2})$
\[
\begin{split}
 &\|H_{\Dso}\brac{v^i, \langle \vec{v},\Dso (\vec{u}-\vec{v})\rangle}\|_{L^2(\R^d)}\\
 \overset{\eqref{eq:comm:5}}{\aleq}&\|\Ds{1-\sigma} v^i\|_{L^{\frac{2d}{1-2\sigma}}(\R^d)}\, \|\Ds{\sigma} \langle \vec{v},\Dso (\vec{u}-\vec{v})\rangle\|_{L^{\frac{2d}{d-1+2\sigma}}(\R^d)}\\
 \aleq&\|\Dso v^i\|_{L^{2d}(\R^d)}\, \|\nabla (\vec{u}-\vec{v})\|_{L^2(\R^d)}\, \brac{\|\Ds{1+\sigma} \vec{u}\|_{L^{\frac{2d}{2(1+\sigma)-1}}(\R^d)} + \|\Ds{1+\sigma} \vec{v}\|_{L^{\frac{2d}{2(1+\sigma)-1}}(\R^d)}}.
 \end{split}
\]
For $\alpha \geq  1+\sigma$ this gives the correct estimate.

For the second term we can estimate with the help of \Cref{la:commies3b}. Taking there $\alpha_1 = \alpha_2 = \frac{1+\beta}{2}$ for a any small $\beta \in (0,1)$, and 
$p_1 = p_2 = \frac{2d}{\beta}$, $p_3 = \frac{2d}{d-2\beta}$ we have $\frac{dp_3}{d+\beta p_3} = 2$, and thus
\[
 \begin{split}
&\brac{\int_{\R^d} \abs{\int_{\R^d} \frac{\brac{v^i(x)-v^i(y)}\, \scpr{\vec{v}(x)-\vec{v}(y)}{\Dso (\vec{u}-\vec{v}) (y)  }}{|x-y|^{d+1}} dy}^2 dx  }^{\frac{1}{2}}\\
\aleq& \|\Ds{\frac{1+\beta}{2}} v^i\|_{L^{\frac{2d}{\beta}} (\R^d)}\, \|\Ds{\frac{1+\beta}{2}} \vec{v}\|_{L^{\frac{2d}{2(1+\beta)-1}}(\R^d)}\, \|\Dso (\vec{u}-\vec{v})\|_{L^{2}(\R^d)}\\
\aleq&\|\Dso v^i\|_{L^{2d} (\R^d)}\, \|\Dso \vec{v}\|_{L^{2d}(\R^d)}\, \|\Dso (\vec{u}-\vec{v})\|_{L^{2}(\R^d)}.
 \end{split}
\]
We can conclude.
 \end{proof}

From the terms in \eqref{eq:verylastpart:split1} it remains to understand the last one. For this we use the halfwave equation of $\vec{u}$ and $\vec{v}$, \eqref{eq:half-wavemapseq}, to write
\begin{equation}\label{eq:verylastpart:split2}
\begin{split}
 &\int_{\R^d} \sum_{i,j=1}^3 v^jH_{\Dso}(v^j ,\Dso{} (u-v)^i)\, \partial_t (u-v)^i\\
=&\int_{\R^d} \sum_{i,j=1}^3 v^jH_{\Dso}(v^j ,\Dso{} (u-v)^i)\, \brac{(\vec{u}-\vec{v})\wedge \Dso \vec{u}}^i\\
&+\int_{\R^d} \sum_{i,j=1}^3 v^jH_{\Dso}(v^j ,\Dso{} (u-v)^i)\, \brac{\vec{v}\wedge \Dso \brac{\vec{u}-\vec{v}}}^i.
 \end{split}
 \end{equation}
In order to estimate the first term on the right-hand side of \eqref{eq:verylastpart:split2} we first establish the following.

Denote with $\lapms{\sigma} \equiv \Ds{-\sigma}$ the Riesz potential, then we have the following estimate. Observe the power of $\|\Dso{} \vec{v}\|_{L^{2d}(\R^d)}$ which is crucial here.

\begin{lemma}\label{la:235236}
Let $\sigma \in (0,1)$, and  $d \geq 2$, then
\[
\begin{split}
 &\|\sum_{j=1}^3\lapms{\sigma} \brac{v^jH_{\Dso}(v^j ,\Dso{} (u-v)^i)} \|_{L^{\frac{2d}{d+1-2\sigma}}(\R^d)}\\
 \aleq &\|\nabla (\vec{u}-\vec{v})\|_{L^2(\R^d)}\ \|\Dso{} \vec{v}\|_{L^{2d}(\R^d)}.
\end{split}
 \]
\end{lemma}
\begin{proof}
The main problem we need to solve is that the term $\Dso{} (u-v)^i$ can not afford any more derivatives.The idea is to factor out $\Ds{\sigma}$ derivatives from this term and absorb it into $\lapms{\sigma}$ -- up to several error terms.

We observe first the following algebraic identity, 
\[
 \begin{split}
&H_{\Dso} (a,\Ds{\sigma} b)  \\
 =& - H_{\Dso} (\Ds{\sigma} a,  b)\\
 &+ \Ds{\sigma} H_{\Dso}(a,b) \\
 &- \Dso H_{\Ds{\sigma}}(a,b) + H_{\Ds{\sigma}} (\Dso a, b) + H_{\Ds{\sigma}} (a, \Dso b).
 \end{split}
\]
The last term is a double Leibniz type commutator, which we are going to name $\tilde{H}_{\Dso, \Ds{\sigma}}$
\[
 \tilde{H}_{\Dso,\Ds{\sigma}}(a,b) := \Dso H_{\Ds{\sigma}}(a,b) - H_{\Ds{\sigma}} (\Dso a, b) - H_{\Ds{\sigma}} (a, \Dso b).
\]
We apply the above identity to $a := v^j$, $b := \Ds{1-\sigma} (u-v)^i$. Using Sobolev inequality, \Cref{la:sobolev}, we find
\[
\begin{split}
 &\|\sum_{j=1}^3\lapms{\sigma} \brac{v^jH_{\Dso}(v^j ,(\Dso{} (u-v)^i))} \|_{L^{\frac{2d}{d+1-2\sigma}}(\R^d)}\\
 \aleq &\|\sum_{j=1}^3\brac{v^j H_{\Dso} (\Ds{\sigma} v^j,  \Ds{1-\sigma} (u-v)^i)} \|_{L^{\frac{2d}{d+1}}(\R^d)}\\
 &+\|\sum_{j=1}^3\lapms{\sigma} \brac{v^j \Ds{\sigma} H_{\Dso} (v^j,  \Ds{1-\sigma} (u-v)^i)} \|_{L^{\frac{2d}{d+1-2\sigma}}(\R^d)}\\
 &+\|\sum_{j=1}^3\brac{v^j \tilde{H}_{\Dso,\Ds{\sigma}} (v^j,  \Ds{1-\sigma} (u-v)^i)} \|_{L^{\frac{2d}{d+1}}(\R^d)}.
\end{split}
 \]
The middle term on the right-hand side can be treated further, using that $\lapms{\sigma} \Ds{\sigma} f= f$, and using Sobolev inequality, \Cref{la:sobolev}, we have
\[
 \begin{split}
   &\|\sum_{j=1}^3\lapms{\sigma} \brac{v^j \Ds{\sigma} H_{\Dso} (v^j,  \Ds{1-\sigma} (u-v)^i)} \|_{L^{\frac{2d}{d+1-2\sigma}}(\R^d)}\\
   \aleq&\|\sum_{j=1}^3 v^j  H_{\Dso} (v^j,  \Ds{1-\sigma} (u-v)^i) \|_{L^{\frac{2d}{d+1-2\sigma}}(\R^d)}\\
   &+\|\sum_{j=1}^3\Ds{\sigma} v^j\, H_{\Dso} (v^j,  \Ds{1-\sigma} (u-v)^i) \|_{L^{\frac{2d}{d+1}}(\R^d)}\\
   &+\|\sum_{j=1}^3\lapms{\sigma} H^\ast_{\Ds{\sigma}}\brac{v^j, H_{\Dso} (v^j,  \Ds{1-\sigma} (u-v)^i)} \|_{L^{\frac{2d}{d+1-2\sigma}}(\R^d)}.
 \end{split}
\]
where 
\begin{equation}\label{eq:Hast}
 H^\ast_{\Ds{\sigma}}(f,g) := f\, \Ds{\sigma} g - \Ds{\sigma} f\, g - \Ds{\sigma} (fg)
\end{equation}
is the formal adjoint to $H_{\Ds{\sigma}}$ (see the estimate of this term below).
In summary, we have  
 \begin{equation}\label{eq:235236:split1}
\begin{split}
 &\|\sum_{j=1}^3\lapms{\sigma} \brac{v^jH_{\Dso}(v^j ,(\Dso{} (u-v)^i))} \|_{L^{\frac{2d}{d+1-2\sigma}}(\R^d)}\\
 \aleq &\|\sum_{j=1}^3\brac{v^j H_{\Dso} (\Ds{\sigma} v^j,  \Ds{1-\sigma} (u-v)^i)} \|_{L^{\frac{2d}{d+1}}(\R^d)}\\
&+\|\sum_{j=1}^3 v^j  H_{\Dso} (v^j,  \Ds{1-\sigma} (u-v)^i) \|_{L^{\frac{2d}{d+1-2\sigma}}(\R^d)}\\
   &+\|\sum_{j=1}^3\Ds{\sigma} v^j\, H_{\Dso} (v^j,  \Ds{1-\sigma} (u-v)^i) \|_{L^{\frac{2d}{d+1}}(\R^d)}\\
   &+\|\lapms{\sigma} \sum_{j=1}^3H^\ast_{\Ds{\sigma}}\brac{v^j, H_{\Dso} (v^j,  \Ds{1-\sigma} (u-v)^i)} \|_{L^{\frac{2d}{d+1-2\sigma}}(\R^d)}\\
 &+\|\sum_{j=1}^3\brac{v^j \tilde{H}_{\Dso,\Ds{\sigma}} (v^j,  \Ds{1-\sigma} (u-v)^i)} \|_{L^{\frac{2d}{d+1}}(\R^d)}\\
\end{split}
 \end{equation}

We treat \underline{first and second term in \eqref{eq:235236:split1}} at the same time. Namely for $\alpha \in \{0,\sigma\}$ we discuss
\[
 \|\sum_{j=1}^3\brac{v^j H_{\Dso} (\Ds{\alpha} v^j,  \Ds{1-\sigma} (u-v)^i)} \|_{L^{\frac{2d}{d+1-2(\sigma-\alpha)}}(\R^d)}.
\]
We observe  by \Cref{la:Hdssformula}, for any $x \in \R^d$
\[
\begin{split}
 &\sum_{j=1}^3 v^j(x) H_{\Dso} (\Ds{\alpha} v^j,  \Ds{1-\sigma} (u-v)^i)(x) \\
 =&c\sum_{j=1}^3 v^j(x) \int_{\R^d} \frac{(\Ds{\alpha} v^j(x)-\Ds{\alpha} v^j(y)) \,  \brac{\Ds{1-\sigma} (u-v)^i(x) - \Ds{1-\sigma} (u-v)^i(y)}}{|x-y|^{d+1}}\,  dy \\
 =&c\int_{\R^d} \frac{(\langle \vec{v},\Ds{\alpha} \vec{v}\rangle(x)-\langle \vec{v},\Ds{\alpha} \vec{v}\rangle(y)) \,  \brac{\Ds{1-\sigma} (u-v)^i(x) - \Ds{1-\sigma} (u-v)^i(y)}}{|x-y|^{d+1}}\,  dy \\
 &+c\sum_{j=1}^3 \int_{\R^d} \frac{\Ds{\alpha} v^j(y)\, \brac{v^j(x)-v^j(y)} \,  \brac{\Ds{1-\sigma} (u-v)^i(x) - \Ds{1-\sigma} (u-v)^i(y)}}{|x-y|^{d+1}}\,  dy \\
 =&c\int_{\R^d} \frac{(\langle \vec{v},\Ds{\alpha} \vec{v}\rangle(x)-\langle \vec{v},\Ds{\alpha} \vec{v}\rangle(y)) \,  \brac{\Ds{1-\sigma} (u-v)^i(x) - \Ds{1-\sigma} (u-v)^i(y)}}{|x-y|^{d+1}}\,  dy \\
 &+c\sum_{j=1}^3 \Ds{\alpha} v^j(x)\, \int_{\R^d} \frac{\brac{v^j(x)-v^j(y)} \,  \brac{\Ds{1-\sigma} (u-v)^i(x) - \Ds{1-\sigma} (u-v)^i(y)}}{|x-y|^{d+1}}\,  dy \\
 &-c\sum_{j=1}^3 \int_{\R^d} \frac{\brac{\Ds{\alpha} v^j(x)-\Ds{\alpha} v^j(y)}\, \brac{v^j(x)-v^j(y)} \,  \brac{\Ds{1-\sigma} (u-v)^i(x) - \Ds{1-\sigma} (u-v)^i(y)}}{|x-y|^{d+1}}\, dy \\
 \end{split}
\]
That is, again using  \Cref{la:Hdssformula},
\begin{equation}\label{eq:235236:split2}
\begin{split}
 &\abs{\sum_{j=1}^3 v^j(x) H_{\Dso} (\Ds{\alpha} v^j,  \Ds{1-\sigma} (u-v)^i)(x)}\\
 \aleq&\abs{H_{\Dso} \brac{\langle \vec{v},\Ds{\alpha} \vec{v}\rangle, \Ds{1-\sigma} (u-v)^i} } \\
  &+\max_{j} \abs{\Ds{\alpha} v^j} \abs{H_{\Dso} \brac{v^j, \Ds{1-\sigma} (u-v)^i} } \\
 &+\int_{\R^d} \frac{\abs{\Ds{\alpha} \vec{v}(x)-\Ds{\alpha} \vec{v}(y)}\, \abs{\vec{v}(x)-\vec{v}(y)} \,  \abs{\Ds{1-\sigma} (u-v)^i(x) - \Ds{1-\sigma} (u-v)^i(y)}}{|x-y|^{d+1}}\,  dy.
 \end{split}
\end{equation}
If $\alpha = 0$ the first term in \eqref{eq:235236:split2} is zero. Otherwise we have $\alpha = \sigma$ and applying twice Leibniz rule estimates,
\[
 \begin{split}
  &\|H_{\Dso} \brac{\langle \vec{v},\Ds{\sigma} \vec{v}\rangle, \Ds{1-\sigma} (u-v)^i} \|_{L^{\frac{2d}{d+1}}(\R^d)}\\
 \overset{\eqref{eq:comm:5}}{\aleq}& \|\Ds{1-\sigma} \langle \vec{v},\Ds{\sigma} \vec{v}\rangle\|_{L^{2d}(\R^d)}\, \|\Dso (u-v)^i\|_{L^{2}(\R^d)}\\
 \overset{\eqref{eq:ucdotu}}{\aeq}&\|\Ds{1-\sigma} H_{\Ds{\sigma}}(\vec{v}\cdot,\vec{v})\|_{L^{2d}(\R^d)}\, \|\Dso (u-v)^i\|_{L^{2}(\R^d)}\\
 \overset{\eqref{eq:comm:123}}{\aleq}&\|\Ds{1-\frac{\sigma}{2}} \vec{v} \|_{L^{\frac{2d}{1-\frac{\sigma}{2}}}(\R^d)}\, \|\Ds{\frac{\sigma}{2}} \vec{v} \|_{L^{\frac{4d}{\sigma}}(\R^d)}\, \|\Dso (u-v)^i\|_{L^{2}(\R^d)}
 \end{split}
\]
For sufficiently small $\sigma > 0$ we can apply Gagliardo-Nirenberg inequality, \Cref{la:GagliardoNirenberg}, and obtain
\[
 \|\Ds{1-\frac{\sigma}{2}} \vec{v} \|_{L^{\frac{2d}{1-\frac{\sigma}{2}}}(\R^d)}\, \|\Ds{\frac{\sigma}{2}} \vec{v} \|_{L^{\frac{4d}{\sigma}}(\R^d)} 
 \aleq \|\vec{v}\|_{L^\infty(\R^d)}\, \|\Dso \vec{v}\|_{L^{2d}(\R^d)}.
\]
This settles the first term in \eqref{eq:235236:split2}.

For the second term in \eqref{eq:235236:split2} we estimate if $\alpha = 0$,
\[
\begin{split}
 &\max_{j} \|\abs{v^j} \abs{H_{\Dso} \brac{v^j, \Ds{1-\sigma} (u-v)^i} }\|_{L^{\frac{2d}{d+1-2(\sigma-\alpha)}}(\R^d)}\\
 \aleq& \max_{j} \|v^j\|_{L^{\infty}(\R^d)}\, \|H_{\Dso} \brac{v^j, \Ds{1-\sigma} (u-v)^i} \|_{L^{\frac{2d}{d+1-2\sigma}}(\R^d)}\\
 \overset{\eqref{eq:comm:5}}{\aleq}&\max_{j} \|v^j\|_{L^{\infty}(\R^d)}\, \|\Ds{1-\sigma} v^j \|_{L^{\frac{2d}{1-2\sigma}}(\R^d)}\, \|\Dso (u-v)^i\|_{L^2(\R^d)}\\
 \aleq&\max_{j} \|v^j\|_{L^{\infty}(\R^d)}\, \|\Dso v^j \|_{L^{2d}(\R^d)}\, \|\Dso (u-v)^i\|_{L^2(\R^d)}.
 \end{split}
\]
In the last line we used Sobolev inequality, \Cref{la:sobolev}.
If $\alpha = \sigma$ we adapt this slightly,
\begin{equation}
\label{eq:est:aargh34325}
\begin{split}
 &\max_{j} \|\abs{\Ds{\sigma} v^j} \abs{H_{\Dso} \brac{v^j, \Ds{1-\sigma} (u-v)^i} }\|_{L^{\frac{2d}{d+1}}(\R^d)}\\
 \aleq& \max_{j} \|\Ds{\sigma} v^j\|_{L^{\frac{2d}{\sigma}}(\R^d)}\, \|H_{\Dso} \brac{v^j, \Ds{1-\sigma} (u-v)^i} \|_{L^{\frac{2d}{d+1-\sigma}}(\R^d)}\\
 \overset{\eqref{eq:comm:5}}{\aleq}&\max_{j} \|\Ds{\sigma} v^j\|_{L^{\frac{2d}{\sigma}}(\R^d)}\, \|\Ds{1-\sigma} v^j \|_{L^{\frac{2d}{1-\sigma}}(\R^d)}\, \|\Dso (u-v)^i\|_{L^2(\R^d)}\\
 \aleq&\max_{j} \|v^j\|_{L^{\infty}(\R^d)}\, \|\Dso v^j \|_{L^{2d}(\R^d)}\, \|\Dso (u-v)^i\|_{L^2(\R^d)}\\
 \end{split}
\end{equation}
In the last line we used Gagliardo-Nirenberg inequality, \Cref{la:GagliardoNirenberg}.
This provides the desired estimate for the second term in \eqref{eq:235236:split2}. 

For the third term in \eqref{eq:235236:split2} we use \Cref{la:commies3} (observe that all $p_i \geq 2$, so \eqref{eq:c3:23525} is trivially satisfied). 

If $\alpha = 0$, we instead use \Cref{la:commies3toy},
\[
\begin{split}
 &\left \|\int_{\R^d} \frac{\abs{\vec{v}(x)-\vec{v}(y)}\, \abs{\vec{v}(x)-\vec{v}(y)} \,  \abs{\Ds{1-\sigma} (u-v)^i(x) - \Ds{1-\sigma} (u-v)^i(y)}}{|x-y|^{d+1}}\,  dy \right \|_{L^{\frac{2d}{d+1-2\sigma}}(\R^d)}\\
\aleq&\|\vec{v}\|_{L^\infty(\R^d)}\, \left \|\int_{\R^d} \frac{\abs{\vec{v}(x)-\vec{v}(y)} \,  \abs{\Ds{1-\sigma} (u-v)^i(x) - \Ds{1-\sigma} (u-v)^i(y)}}{|x-y|^{d+1}}\,  dy \right \|_{L^{\frac{2d}{d+1-2\sigma}}(\R^d)}\\
\aleq&\|\vec{v}\|_{L^\infty(\R^d)}\,  \|\Ds{1-\sigma} \vec{v}\|_{L^{\frac{2d}{1-2\sigma}(\R^d)}}\, \|\Dso (u-v)^i\|_{L^{2}(\R^d)}\\
\aleq&\|\vec{v}\|_{L^\infty(\R^d)}\,  \|\Dso \vec{v}\|_{L^{2d}(\R^d)}\, \|\Dso (u-v)^i\|_{L^{2}(\R^d)}
 \end{split}
\] 
If $\alpha = \sigma$,
\[
\begin{split}
 &\left \|\int_{\R^d} \frac{\abs{\Ds{\sigma} \vec{v}(x)-\Ds{\sigma} \vec{v}(y)}\, \abs{\vec{v}(x)-\vec{v}(y)} \,  \abs{\Ds{1-\sigma} (u-v)^i(x) - \Ds{1-\sigma} (u-v)^i(y)}}{|x-y|^{d+1}}\,  dy \right \|_{L^{\frac{2d}{d+1}}(\R^d)}\\
 \aleq&   \|\Ds{2\sigma} \vec{v} \|_{L^{\frac{2d}{2\sigma}}(\R^d)}\, 
 \|\Ds{1-2\sigma} \vec{v}\|_{L^{\frac{2d}{1-2\sigma}}(\R^d)}\, 
 \|\Dso (u-v)^i\|_{L^2(\R^d)}\\
 \aleq&\|\vec{v}\|_{L^\infty(\R^d)}\, \|\Ds{1} \vec{v}\|_{L^{2d}(\R^d)}\, \|\Dso (u-v)^i\|_{L^2(\R^d)}.
 \end{split}
\] 
This provides the desired estimates for the terms in \eqref{eq:235236:split2}, i.e. the estimates for first and second term in \eqref{eq:235236:split1}.

The \underline{third term in \eqref{eq:235236:split1}}, has already been estimated in \eqref{eq:est:aargh34325}.

The \underline{fourth term in \eqref{eq:235236:split1}} we treat by duality. Namely, for some $\psi \in C_c^\infty(\R^d)$, $\|\psi\|_{L^{\frac{2d}{d+2\sigma-1}}(\R^d)} \leq 1$ we have, using also integration by parts, and for some $\gamma < \sigma$,
\[
 \begin{split}
&   \|\lapms{\sigma} H^\ast_{\Ds{\sigma}}\brac{v^j, H_{\Dso} (v^j,  \Ds{1-\sigma} (u-v)^i)} \|_{L^{\frac{2d}{d+1-2\sigma}}(\R^d)}\\
\aleq&\int_{\R^d} H^\ast_{\Ds{\sigma}}\brac{v^j, H_{\Dso} (v^j,  \Ds{1-\sigma} (u-v)^i)}\, \lapms{\sigma} \psi\\
\overset{\eqref{eq:Hast}}{=}& \int_{\R^d}\, H_{\Dso} (v^j,  \Ds{1-\sigma} (u-v)^i)\, H_{\Ds{\sigma}} \brac{v^j, \lapms{\sigma} \psi}\\
\aleq& \|\Ds{1-\sigma} v^j\|_{L^{\frac{2d}{1-\sigma+\gamma-\sigma}}(\R^d)} \| \Dso (u-v)^i)\|_{L^2(\R^d)}\, \|\Ds{\gamma} v^j\|_{L^{\frac{2d}{\gamma}}(\R^d)} \|\lapms{\gamma} \psi\|_{L^{\frac{2d}{d-1+2(\sigma-\gamma)}}(\R^d)}\\
\aleq& \|v^j\|_{L^\infty(\R^d)}^{\gamma}\, \|\Dso v^j\|_{L^{2d}(\R^d)}^{1-\gamma} 
\| \Dso (u-v)^i)\|_{L^2(\R^d)}\, \|v^j\|_{L^\infty(\R^d)}^{1-\gamma} \|\Dso v^j\|_{L^{2d}(\R^d)}^{\gamma} \|\psi\|_{L^{\frac{2d}{d-1+2\sigma)}}(\R^d)}\\
\aleq&\|v^j\|_{L^\infty(\R^d)}\, \|\Dso v^j\|_{L^{2d}(\R^d)}\, \|\Dso (u-v)^i\|_{L^2(\R^d)}.
\end{split}
\]
In the second to last step we applied \Cref{co:gagnirsob} observing that since $\gamma \in (0,\sigma)$, we have
\[
 \frac{2d}{1-\sigma} \leq \frac{2d}{1-\sigma+\gamma-\sigma} \leq \frac{2d}{2(1-\sigma)-1}.
\]

The \underline{last term in \eqref{eq:235236:split1}} we estimate via \Cref{la:doublecommie},
\[
\begin{split}
&\|\sum_{j=1}^3\brac{v^j \tilde{H}_{\Dso,\Ds{\sigma}} (v^j,  \Ds{1-\sigma} (u-v)^i)} \|_{L^{\frac{2d}{d+1}}(\R^d)}\\
\aleq&\max_{j =1,2,3}\|v^j\|_{L^\infty(\R^d)}\, \|\Dso v^j\|_{L^{2d}(\R^d)}\, \| \Dso (u-v)^i \|_{L^2(\R^d)}.
\end{split}
\]
This provides the desired estimates for all terms on the right-hand side of \eqref{eq:235236:split1}, and we can conclude.
\end{proof}

\Cref{la:235236} implies control over the first term on the right-hand side of \eqref{eq:verylastpart:split2} 
\begin{lemma}
For any $\alpha > 0$, $d \geq 3$ we have
\[
\begin{split}
& \abs{\int_{\R^d} \scpr{\sum_{j=1}^3\brac{v^j H_{\Dso}(v^i ,(\Dso{} (u-v)^j))-v^jH_{\Dso}(v^j ,(\Dso{} (u-v)^i))}}{(\vec{u}-\vec{v}) \wedge \Dso{} \vec{v}}}\\
  \aleq&\|\nabla (\vec{u}-\vec{v})\|_{L^2(\R^d)}^2\ \|\Ds{\alpha} \vec{v}\|_{L^{\frac{2d}{2\alpha-1}}(\R^d)}^2.
 \end{split}
\]
\end{lemma}
\begin{proof}
For any $\sigma > 0$ we may write, using e.g. the Fourier transform to justify this ``integration by parts'',
 \[
 \begin{split}
 & \abs{\int_{\R^d} \sum_{i,j=1}^3 v^jH_{\Dso}(v^j ,(\Dso{} (u-v)^i))\, \brac{(\vec{u}-\vec{v})\wedge \Dso \vec{u}}^i}\\
=&\abs{\int_{\R^d} \sum_{i,j=1}^3 \lapms{\sigma} \brac{v^jH_{\Dso}(v^j ,(\Dso{} (u-v)^i))}\, \Ds{\sigma}\brac{(\vec{u}-\vec{v})\wedge \Dso \vec{u}}^i}\\
\aleq&\max_{j,i} \|\lapms{\sigma} \brac{v^jH_{\Dso}(v^j ,(\Dso{} (u-v)^i))}_{L^\frac{2d}{d+1-2\sigma}(\R^d)}\, \|\Ds{\sigma}\brac{(\vec{u}-\vec{v})\wedge \Dso \vec{u}}^i\|_{L^{\frac{2d}{d+2\sigma-1}}(\R^d)}\\
\overset{L~\ref{la:235236}}{\aleq}& \|\nabla (\vec{u}-\vec{v})\|_{L^2(\R^d)}\, \|\Dso \vec{v}\|_{L^{2d}(\R^d)}\, \|\Ds{\sigma}\brac{(\vec{u}-\vec{v})\wedge \Dso \vec{u}}^i\|_{L^{\frac{2d}{d+2\sigma-1}}(\R^d)}.
  \end{split}
 \]
On the other hand, by \Cref{la:wugsd8gfus0df} we have 
\[
 \|\Ds{\sigma} \brac{(\vec{u}-\vec{v}) \wedge \Dso{} \vec{v}}\|_{L^{\frac{2d}{d-1+2\sigma}}(\R^d)} \aleq \|\nabla (\vec{u}-\vec{v})\|_{L^2(\R^d)}\, \|\Ds{1+\sigma} \vec{v}\|_{L^{\frac{2d}{2(1+\sigma)-1}}(\R^d)}.
\]
Combining the above estimates, we have shown 
\[
\begin{split}
& \abs{\int_{\R^d} \scpr{\sum_{j=1}^3\brac{v^j H_{\Dso}(v^i ,(\Dso{} (u-v)^j))-v^jH_{\Dso}(v^j ,(\Dso{} (u-v)^i))}}{(\vec{u}-\vec{v}) \wedge \Dso{} \vec{v}}}\\
  \aleq&\|\nabla (\vec{u}-\vec{v})\|_{L^2(\R^d)}^2\, \|\Dso \vec{v}\|_{L^{2d}(\R^d)} \, \|\Ds{1+\sigma} \vec{v}\|_{L^{\frac{2d}{2(1+\sigma)-1}}(\R^d)}.
 \end{split}
\]
This holds for any small $\sigma > 0$, so setting $\alpha := 1+\sigma$ we can conclude.
\end{proof}

The very last term to estimate is the last term on the right-hand side of \eqref{eq:verylastpart:split2} 
\[
 \sum_{i,j=1}^3\int v^jH_{\Dso}(v^j ,\Dso{} (u-v)^i) \  \brac{\vec{v}\wedge  \Dso{} (\vec{u}-\vec{v})}^i
\]
This last needed estimate is given in
 \begin{lemma}
For $d \geq 2$,
 \[
\begin{split}
 \sum_{i,j=1}^3\int_{\R^d} v^jH_{\Dso}(v^j ,(\Dso{} (u-v)^i)) \brac{\vec{v} \wedge  \Dso{} (\vec{u}-\vec{v})}^i \aleq \|\nabla (\vec{u}-\vec{v})\|_{L^2(\R^d)}^2\, \|\Dso{} \vec{v}\|_{L^{2d}(\R^d)}^2
 \end{split}
\]
 \end{lemma}
 \begin{proof}
We can write the term under consideration as a determinant using the well-known formula 
\[
 \vec{a} \cdot (\vec{b} \wedge \vec{c}) = \det (\vec{a}\, |\, \vec{b} |\, \vec{c}).
\]
Applying this to $\vec{a} = \sum_{j=1}^3 v^jH_{\Dso}(v^j ,\Dso{} (\vec{u}-\vec{v}))$, $\vec{b} = \vec{v}$ and $\vec{c} = \Dso \brac{\vec{u}-\vec{v}}$, we can make the algebraic reformulation
  \[
 \begin{split}
 &\sum_{i,j=1}^3\int_{\R^d} v^jH_{\Dso}(v^j ,(\Dso{} (u-v)^i)) \brac{\vec{v} \wedge \Dso{} (\vec{ u-v})}^i \\
= &\int_{\R^d} \det\brac{\sum_{j=1}^3 v^jH_{\Dso}(v^j , \Dso{} \brac{\vec{ u-v}}) \ \Big |\ \vec{ v} \ \Big |\ \Dso{} \brac{\vec{u}-\vec{v}}} \\
= &-\int_{\R^d} \det\brac{\sum_{j=1}^3 v^jH_{\Dso}(v^j , \Dso{} \brac{\vec{ u-v}}) \ \Big |\ \Dso{} \brac{\vec{u}-\vec{v}}\ \Big |\ \vec{ v} \ } \\
= &-\frac{1}{2}\int_{\R^d} \det\brac{\sum_{j=1}^3 v^jH_{\Dso}(v^j , \Dso{} \brac{{\bf u-v}}) \ \Big |\ \Dso{} \brac{\vec{u}-\vec{v}}\ \Big |\ \vec{v} \ } \\
&+\frac{1}{2}\int_{\R^d} \det\brac{\Dso{} \brac{\vec{u}-\vec{v}}\ \Big |\ \sum_{j=1}^3 v^jH_{\Dso}(v^j , \Dso{} \brac{{\bf u-v}})\ \Big |\ \vec{v} \ } 
 \end{split}
\]
Now we use that the determinant with two collinear columns is zero, and expanding $H_{\Dso}$ we find
  \[
 \begin{split}
 &\sum_{i,j=1}^3\int_{\R^d} v^jH_{\Dso}(v^j ,(\Dso{} (u-v)^i)) \brac{\vec{v} \wedge  \Dso{} (\vec{u}-\vec{v})}^i \\
= &-\frac{1}{2}\int_{\R^d} \det\brac{\sum_{j=1}^3 v^j \brac{\Dso{} \brac{v^j \Dso{} \brac{{\bf u-v}}} - \brac{v^j \Dso{} \Dso{} \brac{{\bf u-v}}}}\ \Big |\ \Dso{} \brac{\vec{u}-\vec{v}}\ \Big |\ \vec{v} \ } \\
&+\frac{1}{2}\int_{\R^d} \det\brac{\Dso{} \brac{\vec{u}-\vec{v}}\ \Big |\ \sum_{j=1}^3 v^j \brac{\Dso{} \brac{v^j \Dso{} \brac{{\bf u-v}}} - \brac{v^j \Dso{} \Dso{} \brac{{\bf u-v}}}} \ \Big |\ \vec{v} \ } \\
\overset{|\vec{v}|^2=1}{=} &-\frac{1}{2}\int_{\R^d} \det\brac{ {\sum_{j=1}^3 v^j\Dso{} \brac{v^j \Dso{} \brac{{\bf u-v}}} - \brac{\Dso{} \Dso{} \brac{{\bf u-v}}}}\ \Big |\ \Dso{} \brac{\vec{u}-\vec{v}}\ \Big |\ \vec{v} \ } \\
&+\frac{1}{2}\int_{\R^d} \det\brac{\Dso{} \brac{\vec{u}-\vec{v}}\ \Big |\ \brac{\sum_{j=1}^3 v^j \Dso{} \brac{v^j \Dso{} \brac{{\bf u-v}}} - \brac{\Dso{} \Dso{} \brac{{\bf u-v}}}} \ \Big |\ \vec{v} \ } \\
= &-\frac{1}{2}\int_{\R^d} \det\brac{ {\sum_{j=1}^3 v^j\Dso{} \brac{v^j \Dso{} \brac{{\bf u-v}}}} \ \Big |\ \Dso{} \brac{\vec{u}-\vec{v}}\ \Big |\ \vec{v} \ } \\
 &+\frac{1}{2}\int_{\R^d} \det\brac{ \brac{\Dso{} \Dso{} \brac{{\bf u-v}}}\ \Big |\ \Dso{} \brac{\vec{u}-\vec{v}}\ \Big |\ \vec{v} \ } \\
 &+\frac{1}{2}\int_{\R^d} \det\brac{\Dso{} \brac{\vec{u}-\vec{v}}\ \Big |\ \brac{\sum_{j=1}^3 v^j \Dso{} \brac{v^j \Dso{} \brac{{\bf u-v}}} - \brac{\Dso{} \Dso{} \brac{{\bf u-v}}}} \ \Big |\ \vec{v} \ } \\
\end{split}
\]
To study cancellation via an integration by parts, it is simpler to expand the determinant as a sum. Set $\epsilon_{k\ell m} := -(-1)^{k+\ell+m}$
\[
 \begin{split}
= &-\frac{1}{2}\sum_{k, \ell, m=1}^3 \epsilon_{k \ell m}\int_{\R^d}  {\sum_{j=1}^3 v^j\Dso{} \brac{v^j \Dso{} \brac{u-v}^k}} \ \ \Dso{} \brac{u-v}^\ell\ \ v^m \  \\
 &+\frac{1}{2}\sum_{k, \ell, m=1}^3 \epsilon_{k \ell m}\int_{\R^d}  \Dso{} \Dso{} \brac{u-v}^k\ \ \Dso{} \brac{u-v}^\ell\ \ v^m  \\
 &+\frac{1}{2}\sum_{k, \ell, m=1}^3 \epsilon_{k \ell m} \int_{\R^d} {\Dso{} \brac{u-v}^k}\ \ \brac{\sum_{j=1}^3 v^j \Dso{} \brac{v^j \Dso{} \brac{u-v}^\ell} - \brac{\Dso{} \Dso{} \brac{u-v}^\ell}} \  v^m  
 \end{split}
\]
and perform an integration by parts to factor out the term $\Dso (u-v)^k$,
\[
 \begin{split}
= &-\frac{1}{2}\sum_{k, \ell, m=1}^3 \epsilon_{k \ell m}\sum_{j=1}^3 \int_{\R^d} \Dso{} \brac{u-v}^k \, v^j\, \Dso \brac{v^j\, \Dso{} \brac{u-v}^\ell\,  v^m }\  \\
 &+\frac{1}{2}\sum_{k, \ell, m=1}^3 \epsilon_{k \ell m}\int_{\R^d}  \Dso{} \brac{u-v}^k\, \Dso{} \brac{ \Dso{} \brac{u-v}^\ell\, v^m } \\
 &+\frac{1}{2}\sum_{k, \ell, m=1}^3 \epsilon_{k \ell m} \int_{\R^d} {\Dso{} \brac{u-v}^k}\ \ \brac{\sum_{j=1}^3 v^j \Dso{} \brac{v^j \Dso{} \brac{u-v}^\ell} - \brac{\Dso{} \Dso{} \brac{u-v}^\ell}} \  v^m
 \end{split}
\]
which we regroup into
\[
\begin{split}
= &-\frac{1}{2}\sum_{k, \ell, m=1}^3 \epsilon_{k \ell m}\int_{\R^d}   \Dso{} \brac{u-v}^k \ \brac{\sum_{j=1}^3 \brac{v^j \Dso{} \brac{v^j\, \Dso{} \brac{u-v}^\ell\,  v^m}} -\sum_{j=1}^3 \brac{v^j \Dso{} \brac{v^j\, \Dso{} \brac{u-v}^\ell}}v^m }\  \\
 &+\frac{1}{2}\sum_{k, \ell, m=1}^3 \epsilon_{k \ell m}\int_{\R^d}  \Dso{} \brac{u-v}^k\ \brac{\Dso{}  \brac{\Dso{} \brac{u-v}^\ell\ \ v^m } - \Dso{}  \brac{\Dso{} \brac{u-v}^\ell\  }v^m}.
\end{split}
\]
Setting ${\bf \Gamma} := \Dso{} ({\bf u-v})$ this becomes
we have 
\[
 \begin{split}
= &-\frac{1}{2}\sum_{k, \ell, m=1}^3 \epsilon_{k \ell m}\int_{\R^d}   \Gamma^k \ \brac{\sum_{j=1}^3 v^j \Dso{} \brac{v^j\, \Gamma^\ell\ \ v^m} -\sum_{j=1}^3 v^j \Dso{} \brac{v^j\, \Gamma^\ell}\, v^m }\  \\
 &+\frac{1}{2}\sum_{k, \ell, m=1}^3 \epsilon_{k \ell m}\int_{\R^d}  \Gamma^k\ \brac{\Dso{}  \brac{\Gamma^\ell\, v^m } - \Dso{} \Gamma^\ell\, v^m}\\
= &-\frac{1}{2}\sum_{k, \ell, m=1}^3 \epsilon_{k \ell m}\int_{\R^d}   \Gamma^k \ \sum_{j=1}^3 v^j \brac{H_{\Dso}\brac{v^j\Gamma^\ell,v^m} + v^j \Gamma^\ell \Dso v^m}\\
 &+\frac{1}{2}\sum_{k, \ell, m=1}^3 \epsilon_{k \ell m}\int_{\R^d}  \Gamma^k\ \brac{H_{\Dso}(\Gamma^\ell, v^m) + \Gamma^\ell \Dso v^m}.
\end{split}
\]
Since 
\[
\begin{split}
 &\sum_{k, \ell, m=1}^3 \epsilon_{k \ell m} \Gamma^k \ \sum_{j=1}^3 v^j v^j \Gamma^\ell \Dso v^m =  \det \brac{\vec{\Gamma}\ \Big | \ \sum_{j=1}^3 v^j v^j \vec{\Gamma}\ \Big |\ \Dso \vec{v}} = 0\\
 \end{split}
\]
and 
\[
 \sum_{k, \ell, m=1}^3 \epsilon_{k \ell m}  \Gamma^k\, \Gamma^\ell \Dso v^m = \det\brac{\vec{\Gamma}\ \Big | \ \vec{\Gamma}\ \Big | \  \Dso \vec{v} } = 0
\]
we finally have obtained 
\[
 \begin{split}
 &\sum_{i,j=1}^3\int_{\R^d} v^jH_{\Dso}(v^j ,(\Dso{} (u-v)^i)) \brac{\vec{v} \wedge  \Dso{} (\vec{u}-\vec{v})}^i \\
= &-\frac{1}{2}\sum_{k, \ell, m=1}^3 \epsilon_{k \ell m}\int_{\R^d}   \Gamma^k \sum_j v^j\ H_{\Dso{}}(v^j \Gamma^\ell,v^m)\  \\
 &+\frac{1}{2}\sum_{k, \ell, m=1}^3 \epsilon_{k \ell m}\int_{\R^d}  \Gamma^k\sum_{j=1}^3 v^j v^j \ H_{\Dso}(\Gamma^\ell,v^m) \\
=&-\frac{1}{2}\sum_{k, \ell, m=1}^3 \epsilon_{k \ell m}\int_{\R^d}   \Gamma^k \sum_j v^j\ \brac{H_{\Dso{}}(v^j \Gamma^\ell,v^m) - v^j  H_{\Dso}(\Gamma^\ell,v^m)}\\
\end{split}
\]
Consequently,
\[
\begin{split}
&\abs{\sum_{i,j=1}^3\int_{\R^d} v^jH_{\Dso}(v^j ,(\Dso{} (u-v)^i)) \brac{\vec{v} \wedge  \Dso{} (\vec{u}-\vec{v})}^i} \\
\aleq & \|\vec{\Gamma}\|_{L^2(\R^d)} \|\vec{v}\|_{L^\infty(\R^d)}\, \max_{j,\ell,m}\|H_{\Dso{}}(v^j \Gamma^\ell,v^m) - v^j  H_{\Dso}(\Gamma^\ell,v^m)\|_{L^2(\R^d)}\\
\end{split}
\]
From \Cref{la:commie3cde} for any $\alpha \in (\frac{1}{2},1)$ (since $d \geq 2$ there is no further assumption necessary),
\[
\begin{split}
 \|\brac{H_{\Dso{}}(v^j \Gamma^\ell,v^m) - v^j  H_{\Dso}(\Gamma^\ell,v^m)}\|_{L^2(\R^d)}  \aleq& \|\vec{\Gamma}\|_{L^2(\R^d)}\, \|\Ds{\alpha} \vec{v}\|_{L^{\frac{2d}{2\alpha-1}}(\R^d)}^2\\
 \aleq& \|\vec{\Gamma}\|_{L^2(\R^d)}\, \|\Dso{} \vec{v}\|_{L^{2d}(\R^d)}^2.
 \end{split}
\]
Recalling that $\vec{\Gamma} = \Dso \brac{\vec{u} -\vec{v}}$ we conclude.
% \[
%    \begin{split}
%  &\sum_{i,j=1}^3\int_{\R^d} v^jH_{\Dso}(v^j ,(\Dso{} (u-v)^i)) \brac{\vec{v} \wedge  \Dso{} (\vec{u}-\vec{v})}^i \\
% =&-\frac{1}{2}\sum_{k, \ell, m=1}^3 \epsilon_{k \ell m}\int_{\R^d}   \Gamma^k \sum_j v^j\ \brac{H_{\Dso{}}(v^j \Gamma^\ell,v^m) - v^j  H_{\Dso}(\Gamma^\ell,v^m)}\\
% \aleq&\|\Gamma\|_{L^2(\R^d)} \|v\|_{L^\infty}\, \|\Gamma\|_{L^2(\R^d)}\, \|\Dso{} \vec{v}\|_{L^{2d}(\R^d)}^2\\
% \aeq&\|\nabla (\vec{u}-\vec{v})\|_{L^2(\R^d)}^2\, \|\Dso{} \vec{v}\|_{L^{2d}(\R^d)}^2.
%  \end{split}
% \]
% 
 \end{proof}

\section{Uniqueness: Proof of Theorem~\ref{th:main}}\label{s:proofmain}
\Cref{th:main} is a consequence of \Cref{th:detest} 
\begin{proof}[Proof of \Cref{th:main}] By Gr\"onwall's lemma, the differential inequality $\dot{\mathcal{E}}(t) \leq  \Sigma(t) \mathcal{E}(t)$ implies
\[
 \mathcal{E}(t_0) \aleq \mathcal{E}(0) \, \exp\brac{\int_{0}^{t_0} \Sigma(t)\, dt}.
\]
The expression in the exponential is finite by assumption \eqref{eq:energyassumption} combined with the estimate for $\Sigma(t)$, \eqref{eq:th:detest:est}. 

If $\vec{u}(0) = \vec{v}(0)$ then $\nabla \vec{u}(0) = \nabla \vec{v}(0)$. Moreover since $\vec{u}$ and $\vec{v}$ both solve the halfwave map equation \eqref{eq:half-wavemapseq} we have 
\[
 \brac{\partial_t \vec{u} -\partial_t \vec{v}} \Big |_{t=0} = \vec{u}(0) \wedge \laps{1} \vec{u}(0)-\vec{v}(0) \wedge \laps{1} \vec{v}(0)=0.
\]
Thus $\mathcal{E}(0) = 0$, and we conclude that $\mathcal{E}(t_0) = 0$ for all $t_0 \in (0,T)$. Thus $\vec{u} -\vec{v}$ is a constant on $\R^d \times [0,T] $ -- and again since $\vec{u}(0) = \vec{v}(0)$ we conclude $\vec{u} \equiv \vec{v}$.
\end{proof}

\section{Strichartz estimates}\label{s:strichartz}
The assumptions of \Cref{th:main}, \eqref{eq:energyassumption} are natural in view of the Keel-Tao estimates \cite{KT98}. The following estimate is a consequence of a careful inspection of the arguments in \cite[p.566]{SS02}.
\begin{lemma}
Let $d \geq 4$. Assume for $T > 0$
\begin{equation}\label{eq:strich:ueq}
\begin{cases}
 (\partial_{tt} - \lap )u =h \quad &\text{in $\R^d \times [0,T]$}\\
 u  = f\quad &\text{in $\R^d\times\{0\}$}\\
 \partial_t u = g\quad &\text{in $\R^d \times \{0\}$}\\
 \end{cases}
\end{equation}
then for all $\alpha \in (\frac{1}{2},\frac{d^2-4d+1}{2(d-1)}]$,
\[
\begin{split}
 \|\Ds{\alpha} u(t)\|_{L^2_t L^{(\frac{2d}{2\alpha-1},2)_x}(\R^d) \times (0,T)}  \aleq& \|\Ds{\frac{d}{2}} f\|_{L^2(\R^d)} + \|\Ds{\frac{d}{2}-1} g\|_{L^2(\R^d)}\\
 &+\|\Ds{\frac{d}{2}-1}h\|_{L^1_t L^2_x (\R^d \times (0,T))}.
 \end{split}
\]
\end{lemma}
We sketch the argument for the convenience of the reader.
\begin{proof}
First we assume $h \equiv 0$, i.e. consider solutions to
\begin{equation}\label{eq:strich:veq}
\begin{cases}
 (\partial_{tt} - \lap )v =0 \quad &\text{in $\R^d \times [0,T] $}\\
 v = f\quad &\text{in $\R^d \times \{0\}$}\\
 \partial_t v = g\quad &\text{in $\R^d \times \{0\}$}.
 \end{cases}
\end{equation}
As in \cite[(5.9)]{SS02} from \cite[Corollary 1.3]{KT98} we have the estimate for $\gamma = \frac{d+1}{2(d-1)}$
\[
 \|v\|_{L^2_t L^{\frac{2(d-1)}{d-3},2}_x} \aleq \|\Ds{\gamma} f\|_{L^2(\R^d)} + \|\Ds{\gamma-1} g\|_{L^2(\R^d)}.
\]
The Lorentz space estimate is is from \cite[p.566]{SS02} and follows from interpolation.
Observe that for any $\alpha \geq 0$ the function $\Ds{\alpha} v$ solves 
\[
\begin{cases}
 (\partial_{tt} - \lap )\Ds{\alpha}v =0 \quad &\text{in $\R^d \times [0,T] $}\\
 \Ds{\alpha}v  = \Ds{\alpha}f\quad &\text{in $\R^d \times \{0\}$}\\
 \partial_t \Ds{\alpha}v = \Ds{\alpha} g\quad &\text{in $\R^d \times \{0\}$}\\
 \end{cases}
\]
we obtain 
\[
 \|\Ds{\alpha} v\|_{L^2_t L^{(\frac{2(d-1)}{d-3},2)}_x} \aleq \|\Ds{\gamma+\alpha} f\|_{L^2(\R^d)} + \|\Ds{\gamma+\alpha-1} g\|_{L^2(\R^d)}.
\]
If we choose in the above inequality $\alpha = \frac{d^2-4d+1}{2(d-1)}+1$, then 
\[
 \gamma + \alpha = \frac{d+1}{2(d-1)}+\frac{d^2-4d+1}{2(d-1)}+1=\frac{d}{2}.
\]
and we thus have found the estimate 
\[
 \|\Ds{1+\frac{d^2-4d+1}{2(d-1)}} {v}\|_{L^2_t L^{(\frac{2(d-1)}{d-3},2)}_x} \aleq \|\Ds{\frac{d}{2}} f\|_{L^2(\R^d)} + \|\Ds{\frac{d}{2}-1} g\|_{L^2(\R^d)}.
\]
We observe that for $d\geq 4$ we have $\frac{d^2-4d+1}{2(d-1)} > 0$. Let now $\alpha \in (\frac{1}{2},\frac{d^2-4d+1}{2(d-1)}]$ then
\[
 \alpha - \frac{d}{\frac{2d}{2\alpha-1}} = 1+\frac{d^2-4d+1}{2(d-1)} - \frac{d}{\frac{2(d-1)}{d-3}}, 
\]
and thus by spacial Sobolev embedding, \Cref{la:sobolev},
\[
 \|\Ds{\alpha} v(t)\|_{L^{\frac{2d}{2\alpha-1},2}(\R^d)} \aleq \|\Ds{1+\frac{d^2-4d+1}{2(d-1)}} v(t)\|_{L^{(\frac{2(d-1)}{d-3},2)}}.
\]
After integrating in time we find 
\[
 \|\Ds{\alpha} v\|_{L^2_t L^{(\frac{2d}{2\alpha-1},2)}_x(\R^d \times (0,T))}  \aleq \|\Ds{\frac{d}{2}} f\|_{L^2(\R^d)} + \|\Ds{\frac{d}{2}-1} g\|_{L^2(\R^d)}.
\]
If $h \neq 0$ we conclude by Duhamel's principle and Minkowski's inequality: we keep denoting by $v$ the solution in \eqref{eq:strich:veq}. By Duhamel principle the solution $u$ to \eqref{eq:strich:ueq} can be written as 
\[
 u(x,t) = v(x,t) + \int_0^t H(x,s,t)\, ds,
\]
where $H$ solves
\begin{equation}\label{eq:Heq}
\begin{cases}
 (\partial_{tt} - \lap_x ) H(\cdot,s,\cdot) = 0 \quad &\text{in } \R^d \times (s,T)\\
 H(\cdot,s,s) =0\quad &\text{in $\R^d$}\\
 \partial_t H(\cdot,s,s) = h(\cdot,s) \quad & \text{in $\R^d$}
\end{cases}
\end{equation}
By a slight abuse of notation we identify $H(x,s,t) = \chi_{t>s} H(x,s,t)$.
We then have by Minkowski's inequality
\[
\begin{split}
 &\left \|\int_0^T \Ds{\alpha} H(x,s,t)\, ds\right \|_{L^2_t L^{(\frac{2d}{2\alpha-1},2)}_x (\R^d \times (0,T))}\\
 \aleq& \int_0^T \|\Ds{\alpha} H(x,s,t)\|_{L^2_t L^{(\frac{2d}{2\alpha-1},2)}_x (\R^d \times (s,T))}\, ds
\end{split}
 \]
Since $H$ solves the homogeneous equation \eqref{eq:Heq}, we can use the same estimates as we have obtained for $v$ beforehand, namely we find 
\[
\begin{split}
  &\int_0^T \|\Ds{\alpha} H(x,s,t)\|_{L^2_t L^{(\frac{2d}{2\alpha-1},2)}_x (\R^d \times (s,T))}\, ds \\
  \aleq& \int_0^T \|\Ds{\frac{d}{2}-1} h(x,s) \|_{L^2_x(\R^d)}\, ds\\
  =&\|\Ds{\frac{d}{2}-1} h\|_{L^1_t L^2_x(\R^d)}.
  \end{split}
\]
Combining these estimates with the ones for $v$ we conclude.
\end{proof}

\bibliographystyle{abbrv}%
\bibliography{bib}%

\end{document}